\documentclass[generic,imslayout,twoside,preprint]{imsart}

\usepackage{amsthm,amsmath,natbib}

\startlocaldefs

\usepackage{amssymb}
\usepackage{mathrsfs} 
\usepackage{bbm} 
\usepackage{graphicx}
\usepackage{hyperref}
\usepackage[justification=justified,singlelinecheck=false]{caption}

\newcounter{subsection1}[section]
\newtheorem{defn}[subsection1]{Definition}
\newtheorem{lemma}[subsection1]{Lemma}
\newtheorem{remark}[subsection1]{Remark}
\newtheorem{example}[subsection1]{Example}
\newtheorem{notation}[subsection1]{Notation}
\newtheorem*{notation*}{Notation}
\newtheorem{normaltheorem}[subsection1]{Theorem}

\newtheorem*{lemmaa}{Lemma A}
\newtheorem*{lemmab}{Lemma B}
\def\refyinKst3{A}
\def\refl22{B}

\numberwithin{equation}{section} 
\numberwithin{subsection1}{section}

\newcounter{theorems}
\newtheorem{theorem}[theorems]{Theorem}
\newtheorem{cor}[theorems]{Corollary}

\DeclareMathOperator{\lip}{lip} 
 
\DeclareMathOperator{\diam}{diam}

\def\ra{\Rightarrow} 
\def\to{\rightarrow} 
\def\iff{\Leftrightarrow}
\def\sw{\subseteq} 
\def\mc{\mathcal} 
\def\mb{\mathbb} 
\def\sc{\setminus} 
\def\E{\mb{E}} 
\def\P{\mb{P}}
\def\R{\mb{R}} 
\def\Q{\mb{Q}} 
\def\N{\mb{N}}
\def\~{\sim}
\def\-{\,;\,} 
\def\wt{\widetilde}
\def\qed{$\blacksquare$}
\def\1{\mathbbm{1}}
\def\cadlag{c\`{a}dl\`{a}g}
\def\CPP{{C\kern-.05em\raise.23ex\hbox{+\kern-.05em+}}}

\def\Ka{($\mathscr{K}1$)}
\def\Kb{($\mathscr{K}2$)}
\def\Kc{($\mathscr{K}4$)}
\def\Kg{($\mathscr{K}3$)}

\def\slfv{S$\Lambda$FV}
\def\l{\left}
\def\r{\right}

\allowdisplaybreaks

\endlocaldefs

\begin{document}

\tiny
\color{blue}
\textit{Accepted to the Annals of Probability}
\color{black}
\vspace{1pc}
\normalsize

\begin{frontmatter}

\title{The Segregated $\Lambda$-coalescent}
\runtitle{}

\author{\fnms{Nic} \snm{Freeman}\ead[label=e1]{nicfreeman1209@gmail.com}\thanksref{t1}}
\address{School of Mathematics\\ University Walk \\ Bristol \\ \printead{e1}}
\thankstext{t1}{Supported by an EPSRC Studentship at Oxford University.}

\begin{abstract}
We construct an extension of the $\Lambda$-coalescent to a spatial continuum and analyse its behaviour.
Like the $\Lambda$-coalescent, the individuals in our model can be separated into (i) a dust component and (ii) large blocks of coalesced individuals. We identify a five phase system, where our phases are defined according to changes in the qualitative behaviour of the dust and large blocks. We completely classify the phase behaviour, including necessary and sufficient conditions for the model to come down from infinity.

We believe that two of our phases are new to $\Lambda$-coalescent theory and directly reflect the incorporation of space into our model. Firstly, our semicritical phase sees a null but non-empty set of dust. In this phase the dust becomes a random fractal, of a type which is closely related to iterated function systems. 
Secondly, our model has a critical phase in which the coalescent comes down from infinity gradually during a bounded, deterministic time interval.
\end{abstract}



\end{frontmatter}


\section{Introduction}
\label{introsec}

Coalescent processes are stochastic models in which a collection of particles start out separated and come together over time. Modern coalescent theory began with the coalescent of \cite{K1982}, which was introduced to describe the family trees of individuals sampled from large haploid populations. Kingman's coalescent was generalized, independently but in the same spirit, by \cite{DK1999}, \cite{P1999} and \cite{S1999}. The resulting process became known as the $\Lambda$-coalescent.

We begin with a heuristic description of the $\Lambda$-coalescent.  At time 0 the $\Lambda$-coalescent starts with a countable infinity of particles, with each particle representing an individual from the population. It is usual to label these initial particles with elements of $\N$. Then, at a countable set of random times during $(0,\infty)$, a subset of the currently present particles are selected and these particles come together to form a coalesced block of particles. This coalesced block is thought of as a single new particle and may subsequently be coalesced into even larger blocks of particles. 

The $\Lambda$-coalescent has been studied intensively over the past decade and its behaviour is now well understood. See \cite{B2009} for an introduction to the $\Lambda$-coalescent and its connections to other parts of probability. 

At any time $t>0$, we can divide the particles within the $\Lambda$-coalescent into two types. Firstly, particles that have not been affected by a coalescence event during $[0,t]$. These particles are singletons at time $t$ and, collectively, make up the \textit{dust component} of the coalescent. Secondly, we have large blocks of particles that were coalesced together during $[0,t]$. Each such block contains a non-trivial proportion of the countable infinity of initial particles (and is therefore infinite itself).

It is possible for the particles within the $\Lambda$-coalescent to come together so fast that the dust vanishes instantaneously after time $0$, leaving only finitely many non-singleton blocks. In this case the $\Lambda$-coalescent is said to come down from infinity.

The $\Lambda$-coalescent is exchangeable, which means that its distribution does not change when the labels of the initial particles are permuted. This implies that the $\Lambda$-coalescent is a non-spatial model, since it means that the random forces which cause groups of particles to coalesce do not depend on the labels of the particles involved.

In reality, children begin life close to their parent
 and only travel so far in a single lifetime. Therefore, it is natural to ask if the geographical space in which the population lives has a noticeable effect on the genealogy of the population. This is believed to be the case, see for example \cite{E2008}. 
 
In this article we construct a spatial extension of the $\Lambda$-coalescent in which the ancestral lines of individuals are more likely to coalesce if those individuals lived nearby. Our model behaves similarly to the $\Lambda$-coalescent but sees additional behaviour, notably an extra phase transition that is directly related to the introduction of space. The corresponding extra phase (known as the critical phase) contains behaviour that we believe is new to $\Lambda$-coalescent theory: in this phase our model comes down from infinity gradually over a deterministic, bounded interval of time.

We define our model in Sections \ref{segspacesec} and \ref{modelintrosec} before stating our main results in Section \ref{ressec}. We compare our model and its behaviour to other $\Lambda$-coalescent type models in Sections \ref{compsec}-\ref{slfvsec}. Our results are proved in Sections \ref{modelformaldef}-\ref{nendsec} and a brief outline of the proofs can be found in Section \ref{proofintro}.

\begin{notation*} All the spaces we consider will be metric spaces and we equip them with the corresponding topology and Borel $\sigma$-field. For sets $A$ and $B$, we write $A\uplus B$ for the disjoint union of $A$ and $B$ (that is, $A\uplus B=A\cup B$ with the implication that $A$ and $B$ are disjoint). We write $\bigcup A=\{a\-\exists b\in A, a\in b\}$. If $A$ is a finite set then we write $|A|$ for the cardinality of $A$, with $|A|=\infty$ if $A$ is infinite. We write $\1\l\{\mc{P}\r\}$ for the function which is $1$ if the property $\mc{P}$ holds and $0$ if it does not. We set $\N_0=\N\cup\{0\}$. 
\end{notation*}

\subsection{Segregated Spaces}
\label{segspacesec}

The geographical space of our model, which we call a \textit{segregated space}, is equipped with a tree structure, as follows. This structure will play a central role in the definition of the Segregated $\Lambda$-coalescent.

Let $\mc{S}\in\{2,3,4,\ldots\}$ and set $S=\{1,2,\ldots,\mc{S}\}$. Let $W_n$ be the set of words $w=w_1w_2\ldots w_n$ of length $n$ with letters $w_i\in S$. Set $W_*=\bigcup_{n=0}^\infty W_n$ be the regular $\mc{S}$-ary tree, where $W_0=\{\emptyset\}$ and $\emptyset$ is the empty word\footnote{We also use the symbol $\emptyset$ for the empty set.}. For each $w=w_1w_2\ldots w_n\in W_n$ we write $|w|=n$, with $|\emptyset|=0$. If $n=0$ then we define $w_1\ldots w_n=\emptyset$. If $w=w_1\ldots w_n$ and $i\in S$ then we set $wi=w_1\ldots w_n i\in W_{n+1}$.


\begin{defn}\label{segspacedef}
Let $(K,D_K)$ be a complete metric space, equipped with a family of non-empty measurable subsets $(K_w)_{w\in W_*}$ and a probability measure $\lambda$. We say $K$ is a segregated space if it satisfies:
\begin{samepage}
\begin{itemize}
\item[{\Ka}] $K=K_{\emptyset}$ and for all $w\in W_*$, $K_w=\biguplus_{i\in S} K_{wi}$.
\item[{\Kb}] There exists a sequence $(L_n)_{n=0}^\infty\sw (0,\infty)$ such that $L_n\to 0$ and for all $w\in W_*$, $\max\{D_K(x,y)\-x,y\in K_w\}\leq L_{|w|}.$
\item[{\Kg}] For all $w\in W_*$, $\lambda(K_w)=\mc{S}^{-|w|}$.
\item[{\Kc}] For all $w\in W_*$ there exists $i\in S$ such that $\overline{K_{wi}}\sw K_w$.
\end{itemize}
\end{samepage}
\end{defn}

We say $K_w$ is a complex of $K$ with level $|w|$. If $K_v\sw K_w$ then we say $K_v$ is a subcomplex of $K_w$.

\begin{example}\label{cantorex}
The prototype example of a segregated space is the middle third Cantor set. This is the unique non-empty compact subset $\mc{K}$ of $[0,1]$ which satisfies $F_1(K)\uplus F_2(K)=K$, where $F_1(x)=x/3$ and $F_2(x)=2/3+x/3$. We set $K_\emptyset=K$ and define $K_w$ iteratively by the relation $K_{wi}=F_i(K_w)$. The measure $\lambda$ is the uniform Bernoulli measure on $K$, with $\mc{S}=2$ and $\lambda(K_w)=2^{-|w|}$.
\end{example}

In Example \ref{cantorex}, $K$ is a totally disconnected set of Lebesgue measure zero. This is unnatural from the point of view of population modelling, where it is usual to use $\R^d$ as a model for spatial continua.

\begin{example}\label{01ex}
Let $\mc{S}=2$ and set $K=K_\emptyset=[0,1]$. Note that
$$[0,1]=[0,1/2]\uplus (1/2,1]=[0,1/4]\uplus (1/4,1/2]\uplus (1/2,3/4]\uplus (3/4,1]=\ldots.$$
Set $K_1=[0,1/2]$, $K_2=(1/2,1]$, $K_{11}=[0,1/4]$, $K_{12}=(1/4,1/2]$, $K_{21}=(1/2,3/4]$ and so on. Take $\lambda$ as Lebesgue measure on $[0,1]$. Note that this example is easily adapted to higher dimensions and $\mc{S}\geq 2$.
\end{example}

We will use {\Ka} so frequently that it would be impractical to reference it on every application. However, we will not use the other conditions without explicitly saying so. The purpose of {\Kb} is as follows: suppose $(w_n)$ is a sequence in $W_*$ such that $|w_n|\to\infty$ and $K_{w_{n+1}}\sw K_{w_n}$, then {\Kb} implies that $\cap_n K_{w_n}$ is either empty or equal to a single point. Condition {\Kc} is designed to prevent pathological examples of the sample space and will be discussed further in Section \ref{Osec}.

Initially, each point of $K$ will be the location of precisely one individual. In view of {\Kg}, we think of $\lambda$ as a uniform measure on $K$. The measure $\lambda$ is important to us because it tells us whether a non-empty set of individuals (i.e.~a subset of $K$) comprises a null or positive proportion of the total population. 

\begin{lemma}\label{Kdlemma} For all $w\in W_*$, $\lambda(K_w)>0$. For all $x\in K$, $\lambda(\{x\})=0$.
\end{lemma}
\begin{proof}
The first statement follows trivially from {\Kg}. If $x\in K$ then by {\Ka} for all $n\in\N$ we have $x\in K_w$ for some $w\in W_n$. Since $\uplus_{w\in W_n}K_w=K$ we have $1=\lambda(K)=\mc{S}^n\lambda(K_w)\geq\mc{S}^n\lambda(\{x\})$. Since $n\in\N$ was arbitrary we must have $\lambda(\{x\})=0$.
\end{proof}

\subsection{The Segregated $\Lambda$-coalescent}
\label{modelintrosec}

For the remainder of this article, let $K$ be a segregated space. Let $\lambda_w$ be the restriction of $\lambda$ to $K_w$, defined by $\lambda_w(\cdot)=\frac{\lambda(K_w\cap \cdot)}{\lambda(K_w)}$.

In this section we define our model, which we formulate as a stochastic flow on $K$. The rate of coalescence in our model is controlled by a sequence $(r_n)_{n\in \N_0}$ (recall $\N_0=\N\cup\{0\}$) such that $r_n\geq 0$ for all $n$. To avoid degeneracy, we assume that $r_n>0$ for some $n\in\N_0$.

\textbf{Heuristic definition:} For each $w\in W_*$ the complex $K_w$ is equipped with an exponential clock that rings repeatedly and at rate $r_n$, where $n=|w|$. Informally, if the clock for $K_w$ rings at time $t\in\R$ then all the particles which are in $K_w$ at time $t-$ are coalesced together and jump at time $t$ into a location $p$ that is sampled according to $\lambda_w$. When this occurs we say that a \textit{coalescence event} $(t,w,p)$ has occurred in $K_w$ at time $t$ with parent point $p$. We say the points $x\in K_w$ are affected by the coalescence event. 

We write the resulting flow of particles as $X_{s,t}:K\to K$, where $X_{s,t}(x)$ is the location at time $t$ of that the particle which was at $x$ at time $s<t$. A graphical representation of the above paragraph can be seen in Figure \ref{lintracfig}. 

It is possible to sample the parent points $p$ according to some measure other than $\lambda_w$, but in this article (for brevity) we restrict ourselves to that special case. See \cite{F2012} for a more general mechanism.

\begin{figure}
\includegraphics[height=2in,width=4.9in]{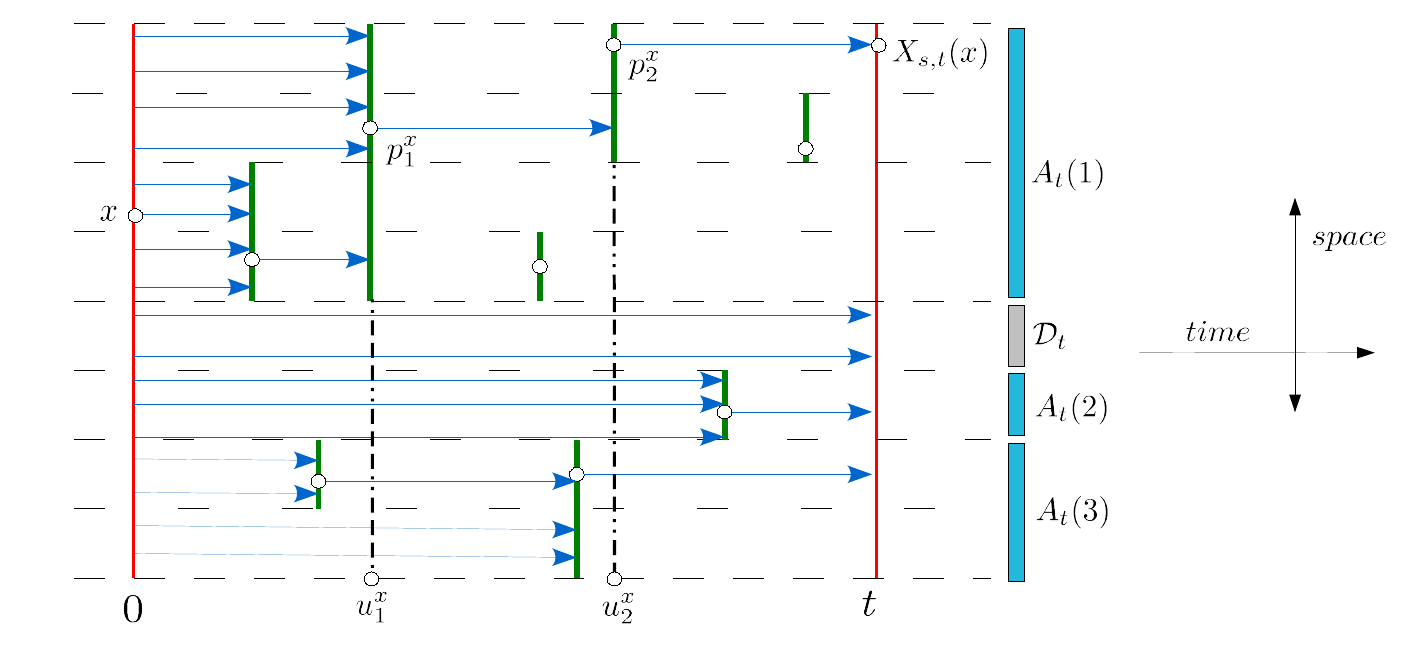} 
\caption{\label{lintracfig} \small The complexes of the geographical space with $\mc{S}=2$ are shown down to level 3, with dotted lines. Coalescence events are shown as thick vertical lines, with parents as circular dots. In this realization there are no coalescence events occurring in complexes of level 4 and above. The movement of some sample particles over $[0,t]$ is shown using arrows. The division of $K$ into dust $\mc{D}_t$ and non-trivial blocks $A_t=\{A_t(1),A_t(2),A_t(3)\}$ (see Section \ref{ressec}) over time $[0,t]$ is also shown. The sequences $(u^x_m)$ and $(p^x_m)$ corresponding to the point $x\in K$ are shown to illustrate Definition \ref{Eseq} (in this example, $N^x_1=1$, $N^x_2=2$, and $N^x_3=\infty$).} 
\end{figure} 

Of course, our heuristic definition only makes mathematical sense if the total number of coalescence events during $[s,t]$ is a.s.~finite for all $s<t$ (equivalently, if the total rate $\sum \mc{S}^nr_n$ of all the exponential clocks is finite). However, it does provides accurate intuition for the behaviour of our model in the general case. For arbitrary $(r_n)$, a mathematical definition which formalizes this intuition can be found in Section \ref{modelformaldef}. We state our existence theorem below, with the understanding that it applies to the definition in Section \ref{modelformaldef}. The proof appears in Section \ref{existproof}.

Let $\mb{M}$ be the space of functions mapping $K$ to itself, equipped with the metric $||f,g||_{\infty}=\sup\{D_K(f(x),g(x))\-x\in K\}.$ 

\begin{theorem}\label{existthm}
For each $s\leq t$, $X_{s,t}$ is an $\mb{M}$-valued random variable. The following properties hold.
\begin{itemize}
\item For all $s\leq t\leq u$, $X_{s,u}=X_{t,u}\circ X_{s,t}$ surely.
\item For all $s\leq t<u\leq v$, $X_{u,v}$ and $X_{s,t}$ are independent.
\item If $t_1-s_1=t_2-s_2$ then $X_{s_1,t_1}$ and $X_{s_2,t_2}$ are identically distributed.
\item For all $t\in \R$ and $x\in K$, $X_{t,t}(x)=x$ surely.
\end{itemize}
\end{theorem}

The formula $X_{s,u}=X_{t,u}\circ X_{s,t}$ is known as the flow property and shows that the population which our model describes has a consistent genealogical structure.

The flow $X$ is time homogeneous and for most of this article we will be interested only in $(X_{0,t})_{t\geq 0}$. We think of each point $x\in K$ being home at time $0$ to a single particle. The function $X_{0,t}$ specifies which particles are coalesced together during $[0,t]$ and where in space the resulting blocks of coalesced particles end up.

\begin{defn}
If $X_{0,t}(K)$ is a finite set then we say the Segregated $\Lambda$-coalescent has come down from infinity at time $t>0$. If $X_{0,t}(K)$ is finite for all $t>0$ then we say the Segregated $\Lambda$-coalescent comes down from infinity.
\end{defn}

Loosely speaking, if our model is to come down from infinity then a large number of coalescence events must occur (i.e.~the rates $r_n$ must be large). On a coalescence event, the coalesced particles jump through space to that events parent location. Thus, if coalescence events occur fast then our particles will also jump fast. The tree structure $(K_w)_{w\in W_*}$ on the segregated space $K$ provides a method of controlling how far particles move when they jump and, as we see in Section \ref{modelformaldef}, is the crucial ingredient that allows us to make sense of our model for any $(r_n)\sw[0,\infty)$.

\subsection{Phases transitions of the Segregated $\Lambda$-coalescent} 
\label{ressec}

For each $t\geq 0$, we say (the individuals which began at) $x$ and $y$ are in the same \textit{block} at time $t$ if $X_{0,t}(x)=X_{0,t}(y)$ and in this case we write $x\sim_t y$. It is easily seen that $\sim_t$ is an equivalence relation on $K$ and we write the equivalence class of $x\in K$ under $\sim_t$ as $[x]_t$. We define
\begin{align*}
\mc{D}_t&=\{x\in K\-[x]_t\text{ is a singleton}\}\\
A_t&=\{[x]_t\-x\in K\sc \mc{D}_t\}.
\end{align*}
Thus, $\mc{D}_t$ is the union of all the singleton blocks and is the \textit{dust component} of our model at time $t$. The set $A_t$ is the set of non-singleton blocks at time $t$. It is easily seen that, for all $t\geq 0$,
\begin{equation}\label{AtDteq}
K=\mc{D}_t\uplus\l(\bigcup A_t\r).
\end{equation}
From the flow property (see Theorem \ref{existthm}) and the definition of $\mc{D}_t$ we have that
\begin{equation}\label{Dmon}
s<t\ra\mc{D}_{t}\sw\mc{D}_{s}.
\end{equation}

As we vary the parameters $(r_n)$ and $\mc{S}$, we say that a phase transition occurs within our model if we see a qualitative change in the behaviour of $\mc{D}_t$ and $A_t$. In particular, we are interested in whether $A_t$ is finite or infinite, and whether $\mc{D}_t$ is empty or non-empty. When $\mc{D}_t$ is non-empty we are also interested in whether $\mc{D}_t$ is $\lambda$-null or has positive measure. Note that $A_0=\emptyset$ and $\mc{D}_0=K$ almost surely.

Let 
\begin{equation}
\tau=\inf\{t\geq 0\-\mc{D}_t=\emptyset\}\label{taudef}
\end{equation}
be the time at which the dust component has been entirely absorbed into the non-singleton blocks. Since $r_N\neq 0$ for some $N\in\N$, at some time $\kappa<\infty$ the exponential clocks associated to all $K_w$ with $|w|=N$ have rung, meaning that $\mc{D}_\kappa=\emptyset$; thus $\tau<\infty$ almost surely.

\begin{theorem}\label{tauthm}
Almost surely, if $t>\tau$ then $\mc{D}_t=\emptyset$ and $A_t$ is finite. Further, if $\tau$ is not almost surely zero then $\P\l[\tau>0\r]=1$.
\end{theorem}

Therefore, we should classify our phases according to the behaviour seen during time $(0,\tau).$ It turns out that our model has five phases, which we now define. Our model is said to be:

\begin{itemize}
\item \textbf{lower subcritical} if: 
\begin{itemize}
\item for all $t\in(0,\infty)$, $\P\l[\tau>t\r]>0$.
\item $\P\l[\forall t\in(0,\tau),\, \lambda(\mc{D}_t)>0\text{ and }|A_t|<\infty\r]=1$.
\end{itemize}

\item \textbf{upper subcritical} if: 
\begin{itemize}
\item for all $t\in(0,\infty)$, $\P\l[\tau>t\r]>0$.
\item $\P\l[\forall t\in(0,\tau),\, \lambda(\mc{D}_t)>0\text{ and }|A_t|=\infty\r]=1$.
\end{itemize}

\item \textbf{semicritical} if: 
\begin{itemize}
\item for all $t\in(0,\infty)$, $\P\l[\tau>t\r]>0$.
\item $\P\l[\forall t\in(0,\tau),\, \lambda(\mc{D}_t)=0\text{ and }|A_t|=\infty\r]=1$.
\end{itemize}

\item \textbf{critical} if there is some (deterministic) $t_0\in(0,\infty)$ such that 
\begin{itemize}
\item $\P\l[\tau<t_0\r]=1$ and for all $t\in(0,t_0)$, $\P\l[t<\tau\r]\in(0,1)$
\item $\P\l[\forall t\in(0,\tau),\, \lambda(\mc{D}_t)=0\text{ and }|A_t|=\infty\r]=1$.
\end{itemize}  
\item \textbf{supercritical} if $\P\l[\tau=0\r]=1$.
\end{itemize}

The quantity $t_0$ is known as the critical time. We are able to completely classify the phases of our model, as follows.

\begin{theorem}\label{phasethm}
Dependent only upon $\mc{S}$ and $(r_n)$, our model is in precisely one of the five phases. In fact, the Segregated $\Lambda$-coalescent $X$ is:
\begin{itemize}
\item lower subcritical if and only if $\sum_n \mc{S}^n r_n<\infty$.
\item upper subcritical if and only if $\sum_n \mc{S}^n r_n=\infty$ and $\sum r_n<\infty$.
\item semicritical if and only if $\sum r_n=\infty$ and $\limsup_n\frac{1}{n}\sum_1^n r_j=0$.
\item critical if and only if $\limsup_n\frac{1}{n}\sum_1^n r_j\in(0,\infty)$
\item supercritical if and only if $\limsup_n\frac{1}{n}\sum_1^n r_j=\infty$.
\end{itemize}
Further, in the critical phase
\begin{equation}\label{t0def}
t_0=\frac{\limsup_{n}\frac{1}{n}\sum_0^nr_j}{\log\mc{S}}.
\end{equation}
\end{theorem}

It follows immediately from Theorem \ref{phasethm} that the Segregated $\Lambda$-coalescent comes down from infinity at $t>0$ if and only if (1) $X$ is supercritical or (2) $X$ is critical and $t\geq t_0$. 

As we commented above, the quantity $\sum\mc{S}^nr_n$ is the total rate of all the exponential clocks involved in the definition of our model. Since each $x\in K$ has $x\in K_w$ for precisely one $w\in W_n$, the quantity $\sum r_n$ is the rate at which coalescence events affect a single point of $K$. It is natural that these two quantities characterise phase transitions. The other quantity which appears in Theorem \ref{phasethm}, $\limsup_n\frac{1}{n}\sum_1^n r_j$, relates to coming down from infinity and will be discussed when we outline our proofs in Section \ref{proofintro}.

As the phase of our model changes, the behaviour of the dust is as expected, in that increasing the intensity of reproduction events reduces the fraction of dust. The lack of monotonicity in the behaviour of the non-singleton blocks is explained as follows. In the lower subcritical phase there are simply not enough events to make anything more than finitely many non-singleton blocks. Then, as the rate increases, there is an intermediate period
where we see a countable infinity of non-trivial blocks. Eventually there are so many reproduction events that they frequently overlap, and we need (a.s.) only finitely many of them to cover $K$.


In the supercritical phase $\tau=0$ almost surely and the question of the behaviour of our model at time $\tau$ is trivial. In the other phases we have the following result.

\begin{cor}\label{taucor}
In all but the supercritical phase, almost surely, $\mc{D}_\tau=\emptyset$ and $A_\tau$ is finite. 
\end{cor}

Theorems \ref{tauthm}, \ref{phasethm} and Corollary \ref{taucor} describe qualitative properties of our model. In particular, they are concerned with behaviour taking place on the tree structure $(K_w)_{w\in W_*}$ and are essentially independent of the metric $D_K$. However, if we examine more detailed properties of our model then $D_K$ does play a role. In Section \ref{nendsec} we describe the behaviour of the Hausdorff dimension of $\mc{D}_t$. Under some quite strong regularity conditions, with $K\sw\R^d$ and $D_K$ equal to the Euclidean metric, we obtain the following result. Let $\chi(t)$ be the Hausdorff dimension of $\mc{D}_t$, conditional on $\mc{D}_t\neq\emptyset$, whenever this is defined.
\begin{itemize}
\item In the lower/upper subcritical and semicritical phases, $\chi(t)=\dim_\mc{H}(K)$ for all $t\geq 0$.
\item In the critical phase, $\chi$ decreases linearly over $[0,t_0]$ with $\chi(0)=\dim_\mc{H}(K)$ and $\chi(t_0)=0$.
\end{itemize}

Proof of the results stated in Sections \ref{modelintrosec} and \ref{ressec} can be found in Sections \ref{modelformaldef}-\ref{nendsec}. An outline of the proofs can be found in Section \ref{proofintro}, but first we will compare our models behaviour to some well known coalescent and population models.


\subsection{Comparison to the $\Lambda$-coalescent} 
\label{compsec}

Let $\Lambda$ be a finite measure on $[0,1]$ and consider the corresponding $\Lambda$-coalescent (see e.g.~\citealt{B2009}). Our model does not feature a spatial analogue of Kingmans coalescent, so we remove the Kingman component of the $\Lambda$-coalescent by specifying that $\Lambda(\{0\})=0$.

If $\Lambda(\{1\})>0$, then the effect on the $\Lambda$-coalescent of the atom at $1$ is as follows; independently of all other mergers, and at rate $\Lambda(\{1\})$, the $\Lambda$-coalescent sees a merger which coagulates the whole population into a single block. Thus, the atom at $1$ serves only to obfuscate the behaviour of the $\Lambda$-coalescent, and is typically removed.

In view of the above paragraph, suppose from now on that $\Lambda(\{1\})=0$. Consider the equivalent of this for the Segregated $\Lambda$-coalescent: we could set $r_0=0$ but if each of the $1$-complexes sees a coalescence event then we face essentially the same issue in that a finite number of coalescence events has covered the whole population. Of course, this could happen with $n$-complexes for any $n\in\N_0$ so in our spatial setting there is no simple way to remove the chance that a finite number of coalescence events may cover the whole population. 

The $\Lambda$-coalescent is said to come down from infinity at time $t>0$ if $\Pi_t$ is a finite set. It is shown in \cite{P1999} that, if the $\Lambda$-coalescent comes down from infinity at some $t>0$ then it does so for all $t>0$, hence the $\Lambda$-coalescent has no equivalent of our critical phase. \citeauthor{P1999}'s proof of this fact uses a zero-one law and relies on the $\Lambda$-coalescent containing only countably many individuals; for this reason the same argument cannot be applied to the Segregated $\Lambda$-coalescent.

Let $\mu^n=\int_0^1x^n\Lambda(dx)$ and consider the case $\mu^{-1}=\infty$. If the $\Lambda$-coalescent does not come down from infinity (e.g. the coalescent of \citealt{BS1998}), then the $\Lambda$-coalescent has empty dust and a countable infinity of non-trivial blocks. This behaviour does not occur in our model. Alternatively, if the $\Lambda$-coalescent does come down from infinity then it has empty dust and finitely many atoms, corresponding to our supercritical phase.

Now consider the case $\mu^{-1}<\infty$. It is shown in \cite{F2012} that, if $\mu^{-2}=\infty$, then the $\Lambda$-coalescent has a countably infinity of non-trivial blocks, and a positive fraction of the population contained within the dust. Similarly, if $\mu^{-2}<\infty$, then it is well known (e.g. Example 19 in \cite{P1999}) that the $\Lambda$-coalescent has only finitely many non-trivial blocks, and has a non-null proportion of the population contained in the dust. Thus the $\Lambda$-coalescent has equivalents of both our upper and lower subcritical phases. 

To summarise the previous few paragraphs, the behaviour seen by our model is that of the $\Lambda$-coalescent, with following modifications.
\begin{enumerate}
\item Playing the role of the cases where $\Lambda(\{1\})>0$, we have the (always positive) probability of having only finitely many non-trivial blocks and no dust.
\item There is no possibility in our model of having a countable infinity of non-trivial blocks and empty dust. This behaviour is replaced by our semicritical phase, in which we see a countably infinity of non-trivial blocks and non-empty null dust.
\item The critical phase appears in between the semi- and supercritical phases.
\end{enumerate}

\subsection{Connections to spatial $\Lambda$-coalescents} 
\label{compspacesec}

Our model is a spatial extension of the $\Lambda$-coalescent. A mean-field version of the $\Lambda$-coalescent has already appeared, in \cite{LS2006}, building on the mean-field version of Kingman's coalescent from \cite{GLW2005}. The model from \cite{LS2006} is referred to in the literature as `the spatial $\Lambda$-coalescent'. 


We refer the reader to \cite{LS2006} for a proper description of their model and restrict ourselves here to outlining some important aspects in which it differs from our own. The model of \citeauthor{LS2006} uses a finite graph $G$ as its geographical space, whereas we use a spatial continuum. The points of its geographical space may be inhabited by more than one block at any time, whereas we permit at most one block of individuals to inhabit a single point of $K$. Further, the blocks of individuals in the model of \citeauthor{LS2006} wander freely around $G$ and may only be coagulated with other blocks that happen to be at the same vertex at the same time. By contrast, blocks in our own model do not move in between mergers, but a single merger involves a non-null proportion of our geographical space. 

Thus, the two models are very different and it is natural to expect different behaviour. In fact, such differences are readily seen. For example, \cite{ABL2010} show that model of \citeauthor{LS2006}, modified slightly so as $G$ is countably infinite and of bounded degree, does not come down from infinity. 

\subsection{Connections to the Spatial $\Lambda$-Fleming-Viot process}
\label{slfvsec}

We now introduce a family of population models that are closely related to the $\Lambda$-coalescent. The dual of the $\Lambda$-coalescent is the $\Lambda$-Fleming-Viot process, constructed in \cite{BL2003} (and also implicitly in \citealt{DK1999}). The $\Lambda$-Fleming-Viot process is a natural generalisation of the classical Fleming-Viot process, which is itself dual to Kingman's coalescent (see e.g.~\citealt{E2011}). The spatial $\Lambda$-Fleming-Viot process ({\slfv}) was introduced in \cite{E2008} as a spatial extension of the $\Lambda$-Fleming-Viot process. 

The {\slfv} process is an infinite system of interacting $\Lambda$-Fleming-Viot processes, with one such process at each site of $\R^d$. See \cite{BEV2010} for a precise description of the {\slfv} process and see \cite{BEV2012} for a survey of recent results concerning (variants of) the {\slfv} process.

Since the dual of the $\Lambda$-Fleming-Viot process is the $\Lambda$-coalescent, and the {\slfv} is a spatial version of the $\Lambda$-Fleming-Viot process, the dual to the {\slfv} process behaves like a spatial version of the $\Lambda$-coalescent. However, as we see below the dual of the {\slfv} process does not display the full range of $\Lambda$-coalescent like behaviour. Our own model, by contrast, shows that coalescent behaviour can be greatly enriched by the introduction of space.

The ancestral lineages of (individuals sampled from a population whose genealogy is described by) the {\slfv} process are compound Poisson processes. Like our own model, these ancestral lineages are only coalesced together when they move through space. However, each such lineage is a compound Poisson process that jumps at finite rate. 

Now, fix $t>0$. Since the geographical of \citeauthor{BEV2010} space is a spatial continuum, we can pick one individual at each point of space and integrate across space (using Fubini's theorem) to see that with positive probability there is a non-null subset of the geographical space, containing an infinite subset of our chosen individuals, all of whom who have not been affected by a reproduction event during $[0,t]$. In the language of coalescents, at all times there is positive probability of the dust being non-empty. Therefore, the model of \cite{BEV2010} does not come down from infinity.

In light of the above paragraph, the reader might wonder why \citeauthor{BEV2012} force their ancestral lineages to be compound Poisson processes rather than using more general L\'{e}vy processes. 
The difficulty stems from an apparent incompatibility between the compensation mechanism (in $\R^d$) of `true' L\'{e}vy processes and the mathematical machinery used to construct the {\slfv} process; we refer the reader to \cite{BEV2010} for a proper discussion. It would not be fair to claim that we have overcome this difficulty in our model. Rather, we choose our state space and reproduction mechanism in such a way as our ancestral lineages can jump at infinite rate, but without the need for L\'{e}vy process style compensation.

Another stochastic model with similar features, a system of interacting Cannings processes on the hierachical group, is investigated by \cite{GHK2012}. They carry out a renormalization procedure corresponding to the hierachical mean field limit, obtaining a detailed description of the limiting object and its behaviour in terms of clustering and coexistance.

\subsection{Outline of the proofs}
\label{proofintro}

Our proof of Theorem \ref{phasethm} comes via several lines of enquiry. Firstly Fubini's theorem produces some useful information and, secondly, the spatial structure of $K$ provides some basic connections between $\mc{D}_t$ and $A_t$. However, the most important contribution comes via a connection between our model and Galton-Watson processes in Varying Environments (GWVEs). A GWVE is a classical Galton-Watson process, with the modification that the offspring distribution of an individual may depend on its generation number. An introduction to GWVEs can be found in \cite{F1972}.

Note that branching structures also play a pivotal role in the study of the $\Lambda$-coalescent, as can be seen in (for example) \cite{BL2000}, \cite{BBC2005} and \cite{BBL2012}. 


For each $w\in W_*$, let $\mc{E}_w$ be the first time $t>0$ at which a coalescence event occurs in the complex $K_w$ (to be clear, the event must occur in precisely $K_w$ and not just inside one of its subcomplexes). We refer to $\mc{E}_w$ as the exponential clock for $K_w$. The connection to our model is as follows. For each $t>0$ and $n\in\N_0$ we define 
\begin{align*}
\mathscr{B}_n^t&=\big\{K_w\-|w|=n \text{ and for all }u\in W_*,\text{ if }K_w\sw K_u\text{ then }\mc{E}_u>t\big\}
\end{align*}
and write $B_n^t=|\mathscr{B}_n^t|$ for the number of elements of $\mathscr{B}^t_n$. Set $\mathscr{B}^t=\cup_{n\in\N_0}\mathscr{B}^t_n$.

It can be seen (in Lemma \ref{Btnlemma}) that $(B^t_n)$ is a GWVE with an $n^{th}$ stage offspring distribution that is Binomial$(S,e^{-tr_{n+1}})$. Note that the case $r_n=c\in(0,\infty)$, where the GWVEs are classical Galton-Watson processes, is part of the critical phase. 

It turns out that the behaviour of $B^t_n=|\mathscr{B}^t_n|$ as $n\to\infty$ is closely connected to the behaviour of $\mc{D}_t$. In fact, Lemma \ref{dustsizebranching} (which appears in Section \ref{dustandgwves}) says that 
\begin{equation}
\mc{D}_t=\bigcap_{n\in\N}\bigcup_{w\in\mathscr{B}^t_n} K_w.\label{iiiii7}
\end{equation}

A GWVE $(B_n)$ is said to be \textit{degenerate} if ${\P\l[\exists n\in\N, B_n=0\r]=1}$. In view of \eqref{iiiii7}, it is important for us to understand when $(B^t_n)$ is degenerate, since in this case $\mc{D}_t=\emptyset$. Conditions equivalent to degeneracy of GWVEs are in general not known, but conditions covering cases including $(B^t_n)$ have been known for some time, in fact since \cite{A1975} and \cite{J1976}. Further conditions were given by \cite{L1992}. The conditions of \cite{J1976} are best suited to our setting and we state them in Lemma \ref{l3}. 

The quantity $\limsup_n\frac{1}{n}\sum_1^n r_j$ (which appeared in Theorem \ref{phasethm}) plays a central role in characterizing degeneracy of $(B^t_n)$. It's precise role is subtle but a partial explanation of the formula is the following. When reproduction events are occurring at a high rate, it becomes common for a larger reproduction event to overwrite the effect of some of the preceding smaller ones. This is borne out by the appearance of the $\limsup$; from Theorem \ref{phasethm} we see that, when $\sum r_n=\infty$, only the $n$-level reproduction events for which $r_n$ is large enough to contribute to $\limsup_n\frac{1}{n}\sum_1^n r_j$ actually take part in determining the phase.  

Formulas similar to \eqref{iiiii7} can be found in the random fractals literature at least as far back as \cite{F1986}, \cite{MW1986} and \cite{G1987} (although these authors did not use branching processes explicitly). Such formulas provide what is now a well known connection between various classes of random fractals and branching processes. In fact, in Section \ref{nendsec} we use a result of \cite{D2009} to calculate the Hausdorff dimension of $\mc{D}_t$, conditional on $\mc{D}_t\neq\emptyset$.

In addition to using GWVEs, in order to understand the behaviour of $\tau$ we will use some techniques from the percolation literature. Many types of branching process, including GWVEs, can be reformulated as an inhomogeneous percolation process on some suitable tree. The relationship is displayed in great generality by, for example, \cite{L1992}. In the case of our GWVEs we have the following.

Consider $W_*$ as the nodes of a regular $\mc{S}$-ary tree, rooted at $\emptyset$ with edge set $\mathbb{E}=\{(w,wi)\-w\in W_*,i\in S\}$. Fix $t>0$. We say that the node $w\in W_*$ is open if $\mc{E}_w>t$ (and closed otherwise) and note that $\mathscr{B}^t$ is the set of points in $W_*$ which are connected to the root node $\emptyset$ via edges with only open nodes at their endpoints. In the language of percolation, $\mathscr{B}^t$ is the open cluster connected to $\emptyset$. Note that each node $w\in W_*$ chooses independently whether it is open or closed. The distribution of $\mc{E}_w$ varies with $|w|$, so in fact $\mathscr{B}^t$ is an inhomogeneous percolation on the $\mc{S}$-ary tree $W_*$.

\section{Existence of the model}
\label{modelformaldef}

In this section we prove the existence of the Segregated $\Lambda$-coalescent. We begin with the definition of our model, formalizing the heuristic description given in Section \ref{modelintrosec}.

Let $\mc{P}$ be the measure on $W_\star\times K$ defined by $\mc{P}(\{w\}\times A)=r_{|w|}\times \lambda_w(A)$
for each $w\in W_*$ and measurable $A\sw K$. Let $(\Omega,\mc{F},\P)$ be a probability space equipped with a Poisson point process $M$ in $\R\times W_*\times K$, of intensity
$dt\otimes \mc{P}(dw, dy)$,
where $dt$ denotes Lebesgue measure. For (measurable) $I\sw\R$, $V\sw W_*$ and $A\sw K$ define
\begin{align*}
M_{I}&=\{(t,w,y)\in M\- t\in I\},\hspace{1pc}M_{I\times V}=\{(t,w,y)\in M\- t\in I, w\in V\}.
\end{align*}
Note that, in terms of $M$, $\mc{E}_w=\inf\{s>0\-M_{(0,s]\times\{w\}}\neq 0\}$.

It will be to our advantage to have some almost sure properties of $M$ as `sure' properties of $M$. In particular, by standard properties of Poisson point processes (see e.g.~\cite{K1993}), with probability one: 
\begin{enumerate}
\item[(a)] For all $k\in\N$ and $n\in \N\cup\{0\}$, $M_{[-k,k]\times W_n}$ is finite.
\item[(b)]For all $u\in\R$, $M_{\{u\}}$ is finite. 
\end{enumerate}
With slight abuse of notation, we simply redefine $M$ so as (a) and (b) hold for all $\omega \in\Omega$.

Let us examine Figure \ref{lintracfig} and determine which coalescence events, according to our heuristic, actually influence the final position of the lineages. Consider an event $(s,w,y)$ in a complex $K_w$ of level $|w|=n$ at time $s\in (0,t)$. The event had no effect on the position at time $t$ of any of the lineages if:
\begin{itemize}
\item There was an event $(s',w',y')$ such that $s<s'<t$ and $K_w\sw K_{w'}$.
\item Or, the final event $(s',w',y')$ such that $0<s'<s$ and $K_w\subset K_{w'}$ had $y'\notin K_w$.
\end{itemize}
Hence, to work out where $x\in K$ should be mapped to over time $[s,t]$ we need only consider the following sequence of events. 

First, look for the final level $0$ event during $(s,t]$ which affected the point $x$. If we find one, say $(u_1,w_1,p_1)$, we then look for the final level $1$ event which was after time $u_1$ and affected $p_1$, and so on. If at any point we don't find a level $n$ event, we simply move up to level $n+1$ and look there. In symbols:

\begin{defn}\label{Eseq}
Fix $(x,s,t)$ with $x\in K$ and $s\leq t$. Let $(u_0,w_0,p_0)=(s,\emptyset, x)$ and set $N_0=-1$. For as long as $N_{m+1}<\infty$ define inductively a pair of (possibly finite) sequences by
\begin{align}
N_{m+1}&=\min\{n> N_m\-\exists (u,w,p)\in {M}_{(u_m,t]\times W_{n}}\text{ such that }p_m\in K_{w}\}\notag\\
E_{m+1}&=(u_{m+1},w_{m+1},p_{m+1})\notag\\
&\hspace{.5pc}\text{ where } u_{m+1}=\max\{u\in(u_m,t]\-(u,w,p)\in {M}_{(u_m,t]\times W_{N_{m+1}}}\text{ and }p_m\in K_{w}\}.\notag
\end{align}
\end{defn}

Define $(E_m)=\{(u_i,w_i,p_i)\-i=1,2\ldots\}$, $(N_m)=\{N_i\-i=1,2\ldots\}$ and note that we do not include the term $i=0$. Define $(u_m),(w_m),(p_m)$ similarly. A graphical demonstration of Definition \ref{Eseq} can be seen in Figure \ref{lintracfig}.

Since (a) and (b) hold for all $\omega\in\Omega$, $(E_m)$ and $(N_m)$ are well defined for all $\omega\in\Omega$. The (finite or infinite) sequence of coalescence events $(E_m)$ contains the only events which affected the final position of the lineage that started from $x$ and moved during time $(s,t]$. 

\begin{notation}[Continuation of Definition \ref{Eseq}]\label{Eseq2} 
The sequence $(E_m)$ depends on $x,s$ and $t$. When we need this distinction (which will be most of the time) we write $$E_m^{x,s,t}=(u^{x,s,t}_m,w^{x,s,t}_m, p^{x,s,t}_m).$$ We write $K_{w^{x,s,t}_m}=K^{x,s,t}_{w_m}$. Occasionally, if $s$ and $t$ are both clear from the context then we may omit them as superscripts and write $E_m^x=(u^x_m,w^x_m,p^x_m)$, $K^x_{w_m}=K_{w^x_m}$.
\end{notation}

We will shortly define $X_{s,t}$ using the language above, but first we need to note a technical point that concerns the following lemma.

\begin{lemma}\label{pmconv}
If the sequence $(E_m^{x,s,t})$ is infinite then the sequence $(p^{x,s,t}_m)$ is convergent. 
\end{lemma}
\begin{proof}
Suppose $(E^{x,s,t}_m)$ is infinite. By Definition \ref{Eseq}, $|w^{x,s,t}_m|$ is $\N$ valued and strictly increasing, so $|w^{x,s,t}_m|\to\infty$ as $m\to\infty$. Thus $K^{x,s,t}_{w_m}$ is a decreasing sequence of sets. Note that $p^{x,s,t}_m\in K_{w_m}$. By {\Kb}, $(p^{x,s,t}_m)$ is a Cauchy sequence so, by completeness of $K$, $(p^{x,s,t}_m)$ is convergent.
\end{proof}

Suppose for a moment that all complexes $K_w$ of $K$ are closed and recall our heuristic definition from Section \ref{modelintrosec}. Then, it makes intuitive sense that reproduction events occurring in complexes $K_{w'}\sw K_w$ cannot move particles in the flow from within $K_w$ into $K\sc K_w$. However, if some $K_{w}$ is not closed then it might be the case that an infinite sequence $(u'_m,w'_m,p'_m)$ of events, with $K_{w'_m}\sw K_w$, could have $\lim p'_m\notin K_w$, because it could be that $\lim p'_m\in \overline{K}_w\sc K_w$. In this case our construction would run into a serious problem; the flow property $X_{s,v}=X_{t,v}\circ X_{s,t}$ would fail. 

To address this issue, we introduce the set
$$\mc{O}=\bigcup\limits_{n\in\N}\bigcup_{\substack{w,w'\in W_n\\w\neq w'}}\overline{K_w}\cap K_{w'}.$$
If $\mc{O}=\emptyset$ then we say $K$ is \textit{completely segregated}. Recall Examples \ref{cantorex} and \ref{01ex} and note that Example \ref{cantorex} is completely segregated but Example \ref{01ex} is not. 

Until further notice, which means until Section \ref{Osec}, we will assume that $K$ is completely segregated i.e.~$\mc{O}=\emptyset$. Then, in Section \ref{Osec} we will discuss the modifications which are necessary to make our arguments work in the case $\mc{O}\neq\emptyset$.

The Segregated $\Lambda$-coalescent is the process $(X_{s,t})_{s\leq t}$ defined as follows.
\begin{equation}\label{pp3}
X_{s,t}(x)=\Bigg\{
\begin{array}{ll}
x & \text{if }(E^{x,s,t}_m)\text{ is empty}\\
p^{x,s,t}_\mc{M} & \text{if }(E^{x,s,t}_m)=(E^{x,s,t}_m)_{m=1}^\mc{M}\text{ for }\mc{M}\in\N\\
\lim\limits_{m\to\infty} p^{x,s,t}_m & \text{if }(E^{x,s,t}_m)\text{ is infinite}.
\end{array}
\end{equation}
In fact, in Section \ref{Osec} we will see that when $\mc{O}\neq\emptyset$ there is a $\P$-null set on which it makes sense to define $X$ differently to \eqref{pp3}. Until then we use \eqref{pp3} for all $\omega\in\Omega$. We now record some results which use the fact that $\mc{O}=\emptyset$.

\begin{lemma}\label{Kwclosed}
Every subcomplex $K_w$ is a closed subset of $K$. 
\end{lemma}
\begin{proof}
Let $n\in\N_0$ and $w\in W_n$. Note that by {\Ka} we have $K=\biguplus_{w\in W_n}K_w$. Hence, if $x\in\overline{K_w}$ then we must have $x\in K_{w'}$ for some $w'\in W_n$. If $x\notin K_w$ then we would have $x\in \overline{K_w}\cap K_u{w'}$, which contradicts $\mc{O}=\emptyset$.
\end{proof}

\begin{remark}\label{K4im}
By Lemma \ref{Kwclosed}, if $\mc{O}=\emptyset$ then {\Kc} holds automatically. 
\end{remark}

\begin{lemma}\label{yinKst2}
For all $s<t$, all $x\in K$ and all $n\in\N$, if $(E_m^{x,s,t})$ is infinite then $\lim\limits_{m\to\infty}p_m^{x,s,t}\in K_{w_n}^{x,s,t}$.
\end{lemma}
\begin{proof}
We noted in the proof of Lemma \ref{pmconv} that $K^{x,s,t}_{w_m}$ is a decreasing sequence of sets. Hence $p_m^{x,s,t}\in K_{w_m}^{x,s,t}\sw K^{x,s,t}_{w_n}$ for all $m\geq n$. By Lemma \ref{Kwclosed}, $K^{x,s,t}_{w_l}$ is closed for all $l\in\N$ and hence $\lim\limits_{m\to\infty}p_m^{x,s,t}\in K_{w_n}^{x,s,t}$.
\end{proof}




\begin{lemma}\label{xeqcomp}
For all $x\in K\sc{\mc{D}_t}$ and $t>0$, 
$[x]_t=K^{x,0,t}_{w_1}$. 
\end{lemma}
\begin{proof}
If $x\notin\mc{D}_t$ then for some $y\in K$ we have $y\neq x$ and $X_{0,t}(x)=X_{0,t}(y)$. It follows from Definition \ref{Eseq}, that $E^{x,s,t}\neq\emptyset$ in this case. Note that if $y\in K^{x,0,t}_{w_1}$ then $X_{0,t}(y)=X_{0,t}(x)$, so $K^{x,0,t}\sw [x]_t$. Now, suppose $z\in K\sc K^{x,0,t}_{w_1}$. If $(K\sc K^{x,0,t}_{w_1})\cap (K\sc K^{z,0,t}_{w_1})\neq \emptyset$ then by {\Ka} either (1) $(K\sc K^{x,0,t}_{w_1})\sw (K\sc K^{z,0,t}_{w_1})$ or (2) $(K\sc K^{z,0,t}_{w_1})\cap (K\sc K^{x,0,t}_{w_1})$. 

If (1) holds then $x\in K^{z,0,t}_{w_1}$ so by Definition \ref{Eseq} we must have $N^{x,0,t}_1\geq N^{z,0,t}_1$, which means $|w_1^{x,0,t}|\geq|w^{z,0,t}_1|$. But combined with (1) and {\Ka} this implies that $K\sc K^{x,0,t}_{w_1}= K\sc K^{z,0,t}_{w_1}$, which contradicts the definition of $z$.

Similarly, if (2) holds then $z\in K^{x,0,t}_{w_1}$ so by Definition \ref{Eseq} we must have  $N^{z,0,t}_{w_1}\geq N^{x,0,t}_{w_1}$, which means $|w_1^{z,0,t}|\geq |w_1^{x,0,t}|$. Combined with (2) and {\Ka} this implies that $K\sc K^{x,0,t}_{w_1}= K\sc K^{z,0,t}_{w_1}$, which again contradicts the definition of $z$.

Thus, we must have $(K\sc K^{x,0,t}_{w_1})\cap (K\sc K^{z,0,t}_{w_1})= \emptyset$. From this, Lemma \ref{yinKst2} implies that $X_{0,t}(x)\neq X_{0,t}(z)$ so as $z\notin [x]_t$. Since $z$ was arbitrary, $[x]_t=K^{x,0,t}_{w_1}$.
\end{proof}

The remainder of Section \ref{modelformaldef} is concerned with proving Theorem \ref{existthm} and establishing some regularity results that we require in Section \ref{phaseproofsec}. Readers who are more interested in proving Theorems \ref{tauthm}, \ref{phasethm} and Corollary \ref{taucor} may wish to omit Sections \ref{existproof}-\ref{meassec} and move straight on to Section \ref{phaseproofsec}.




\subsection{Proof of Theorem \ref{existthm}}
\label{existproof}


The proof comes in three parts, which correspond to the bullet points in the statement of Theorem \ref{existthm}. The first part is a careful check of the flow property. 


\textsc{Part 1.} Let $s\leq t\leq v$ and fix $x\in K$. Write $y=X_{s,t}(x)$. When necessary we will emphasise the dependence with $y=y^{x,s,t}$. We divide into three cases. 

\textbf{If $N^{x,s,t}_1=N^{y,t,v}_1=\infty$:, then} for all $x\in K$, $X_{s,t}(x)=x=y$ and $X_{t,v}(y)=y$. Since $x=y$, $N_{1}^{x,s,v}=\infty$ and $X_{s,t}(x)=X_{t,v}(x)=X_{s,v}(x)$, so $X_{t,v}(X_{s,t}(x))=X_{s,v}(x)$.

\textbf{If $N^{x,s,t}_1=\infty$ and $N^{x,t,v}_1<\infty$, then} $X_{s,t}(x)=x=y$ and hence we must have $u_1^{x,s,t}\geq v$. Hence $(E_m^{x,t,v})=(E_m^{x,s,v})$ and thus $X_{t,v}(X_{s,t}(x))=X_{t,v}(x)=X_{s,t}(x)$. 


\textbf{If $N^{x,s,t}_1<\infty$, then} we have $N^{x,s,v}_1<\infty$. Let 
$$\mc{C}^{s,t,v}=\{x\in K\-\exists m, u^{x,s,v}_m\geq t\}.$$

If $x\notin \mc{C}^{s,t,v}$ then $u_m^{x,s,v}< t$ for all $m$, so from the definitions we have $(E_m^{x,s,t})=(E^{x,s,v}_m)$. Hence $X_{t,v}(x)=X_{s,t}(x)=y$. Suppose it was the case that ${(u_1^{y,t,v},w_1^{y,t,v},p_1^{y,t,v})\in (E^{y,t,v})}$. Note $y\in K_{w_1}^{x,s,v}$ so we must have $(u_1^{y,t,v},w_1^{y,t,v},p_1^{y,t,v})\in (E^{x,s,v}_m)$, which is a contradiction since $u_1^{y,t,v}\geq t$. Hence $(E_m^{y,t,v})$ is empty, and $X_{t,v}=\iota$. Thus, $X_{t,v}(X_{s,t}(x))=X_{s,v}(x)$.  

If $x\in \mc{C}^{s,t,v}$, let $$\mc{M}=\max\{m\- u^{x,s,v}_m< t\}$$ (which is well defined since $(u_m^{x,s,t})$ is strictly increasing), and from the definitions note that $(E_m^{x,s,t})_1^{\mc{M}}=(E_m^{x,s,v})_1^{\mc{M}}$. 

By definition of $\mc{M}$ we have $u^{x,s,v}_{\mc{M}+1}\geq t$ and, since $p^{x,s,v}_\mc{M}=p^{x,s,t}_\mc{M}$, it holds that $N_{\mc{M}+1}^{x,s,v}\leq N_{\mc{M}+1}^{x,s,t}$. Hence $K^{x,s,t}_{w_{\mc{M}+1}}\sw K_{w_{\mc{M}+1}}^{x,s,v}$. By definition, $p^{x,s,t}_\mc{M}\in K^{x,s,t}_{w_{\mc{M}+1}}$ and, we have also that $(K^{x,s,t}_{w_m})$ is decreasing. We have already commented that $K^{x,s,t}_{w_{\mc{M}+1}}\sw K_{w_{\mc{M}+1}}^{x,s,v}$, so it follows from Lemma \ref{yinKst2} that $y^{x,s,t}\in K_{w_{\mc{M}+1}}^{x,s,v}$.

Since both $y$ and $p^{x,s,v}_{\mc{M}}$ are elements of $K^{x,s,v}_{w_{\mc{M}+1}}$, there is no $(u,w,p)\in (E_m^{y,t,v})$ such that $|w|<N_{\mc{M}+1}^{x,s,v}$ - such a $(u,w,p)$ would also have featured in $(E_m^{x,s,v})$, which contradicts the definition of $\mc{M}$. Also, there are no $(u,w,p)\in (E^{y,t,v}_m)$ such that $u>u_{\mc{M}+1}^{x,s,v}$ and $y\in K_w$ - such a $(u,w,p)$ would feature in $(E^{x,s,v}_m)$, which contradicts the definition of $u_{\mc{M}+1}^{x,s,v}$. 

Combining the results of previous two sentences, $(u^{x,s,v}_{\mc{M}+1},w^{x,s,v}_{\mc{M}+1},p^{x,s,v}_{\mc{M}+1})=(u^{y,t,v}_{1},w^{y,t,v}_1,p^{y,t,v}_1)$. Hence $(E_m^{x,s,v})_{m\geq \mc{M}+1}=(E^{x,s,v}_k)_{k\geq 1}$, which implies that $X_{t,v}(y)=X_{s,v}(x)$. This completes the third case.

Since $x$ and $\omega$ were arbitrary, in all cases we have that 
for all $\omega\in\Omega$, $X_{s,v}=X_{t,v}\circ X_{s,t}.$

\textsc{Part 2:} Let $s_1<t_1\leq s_2<t_s$. Since $M_{(s_1,t_1]}$ and $M_{(s_2,t_2]}$ are independent, and the construction of $X_{s,t}$ depended only on ${M}_{s,t}$, it follows immediately that $X_{s_1,t_1}$ and $X_{s_2,t_2}$ are independent.

\textsc{Part 3:} Let $s_1<t_1$ and $s_2<t_2$ with $t_1-s_1=t_2-s_2$. Then $M_{(s_1,t_1]}$ and $\big\{\big(u-(t_2-t_1),w,p\big)\- (u,w,p)\in M_{(s_2,t_2]}\big\}$ are identical in distribution, from which it follows that $X_{s_1,t_1}$ and $X_{s_2,t_2}$ are also identical in distribution. 

\textsc{Part 4:} Let $t\in \R$. Note that $(t,t]$ is empty, so as by Definition \ref{Eseq} we have that $(E^{x,t,t})$ is empty for all $x\in K$. Thus $X_{t,t}$ is the identity function.

\subsection{Regularity}
\label{meassec}

Recall that our underlying Poisson point process $\Pi$ is defined on the probability space $(\Omega,\mc{F},\P)$. 
Throughout this section we denote the dependence on $\omega\in\Omega$ of $X$ by writing $X_{0,t}(\cdot)(\omega)$. Let $\mathbb{B}(K)$ denote the Borel $\sigma$-algebra on $K$ and recall that $D_K$ denotes the metric on $K$.



\begin{lemma}\label{Ksep}
$K$ is separable.
\end{lemma}
\begin{proof}
By Lemma \ref{Kdlemma}, each $K_w$ is non-empty. For each $K_w$ pick some point $x(w)\in K_w$ and define $\mathbb{D}=\{x(w)\-w\in W^*\}$. Note that $\mathbb{D}$ is countable.

Let $O$ be an open set of $K$. Since $K$ is a metric space, for some $r>0$ and $y\in K$, $B_r(y)\sw O$. For $y\in K$ and $n\in\N$ let $K_{(y,n)}$ be the unique complex $K_w$ of $K$ such that $|w|=n$ and $y\in K_w$. By {\Kb}, for some $n\in\N$ we have $r_n<r/2$, so as $K_{(y,n)}\sw B_r(y)\sw O$. By definition of $x(w)$ there is $w\in W_*$ such that $x(w)\in K_{(y,n)}\sw\mc{O}$. Hence $\mathbb{D}$ is a countable dense subset of $K$.
\end{proof}

\begin{lemma}\label{borelK}
The Borel $\sigma$-algebra on $K$ is generated by $(K_w)_{w\in W_*}$.
\end{lemma}
\begin{proof}
By Definition \ref{segspacedef}, each $K_w$ is measurable, so it is clear that $\sigma(K_w\-w\in W_*)\sw\mathbb{B}(K)$. We will now prove the reverse inclusion.

By Lemma \ref{Ksep}, $K$ is separable, hence any open subset of $K$ can be written as a union of only countably many open balls of $K$. Hence $\mathbb{B}(K)$ is generated by the open balls of $K$. So the proof is complete if we can show that any open ball of $K$ is contained in $\sigma(K_w\-w\in W_*)$.

To this end, let $B_r(x)=\{y\in K\-|y-x|<r\}$ be a fixed but arbitrary open ball in $K$. By {\Ka}, for each $y\in K$ and $n\in\N$, let $K_{(y,n)}$ be the unique complex $K_w$ of $K$ such that $|w|=n$ and $y\in K_w$. Note that 
\begin{equation}\label{pppp1001}
B_r(x)\supseteq\bigcup\{K_w\-w\in W_*, K_w\sw B_r(x)\}
\end{equation}
is tautologically true, and, since $W_*$ is countable, the union on the right is countable. Now, suppose that $y\in B_r(x)$. Since $B_r(x)$ is open, for some $\epsilon>0$ we have $B_y(\epsilon)\sw B_r(x)$. By {\Kb}, for some sufficiently large $n\in\N$ we have 
$K_{w(y,n)}\sw B_\epsilon(y)\sw B_r(x).$ 
However, this implies that $K_{w(y,n)}\in \{K_w\-w\in W_*, K_w\sw B_r(x)\}$, so as $y\in\bigcup\{K_w\-w\in W_*, K_w\sw B_r(x)\}$. Hence, in fact 
\eqref{pppp1001} is an equality,
and thus $B_r(x)\in \sigma(K_w\-w\in W_*)$. 
\end{proof}

\begin{lemma}\label{xwmeas}
For each $s<t$, $(x,\omega)\mapsto X_{s,t}(x)(\omega)$ is a measurable function from $K\times\Omega\to K$. 
\end{lemma}
\begin{proof}
The definition of our model is translation invariant across time, so it suffices to consider the case $s=0$. For $v\in W_*$ let $\mc{W}(v)=\{w'\in W_*\-K_{v}\sw K_{w'}\}.$ Fix $w\in W_*$. We note
\begin{align*}
\{(x,&\omega)\in K\times \Omega\- X_{0,t}(x)(\omega)\in K_w\}\notag\\
&=\bigcup\limits_{v\in W_*}K_v\times \{\omega\in\Omega\-X_{0,t}(K_v)\sw K_w\}\notag\\
&=\bigcup\limits_{v\in W_*}K_v\times \bigg[\{\omega\in\Omega\sc\mc{A}\-K_u\sw K_v\}\uplus\Big(\mc{A}\,\cap\\
&\hspace{1pc}\big\{\omega\in\Omega\-\nexists(u',w',p')\in\Pi_{(0,t]}, K_{v}\sw K_{w'}, p'\notin K_w, \Pi_{[u',t]\times\mc{W}(v)}=\emptyset\big\}\Big)\bigg]
\end{align*}
Note that $\{\omega\in\Omega\sc\mc{A}\-K_u\sw K_v\}$ is either empty or equal to the measurable set $\Omega\sc\mc{A}$. From the representation above, it follows that 
$\{(x,\omega)\in K\times \Omega\- X_{0,t}(x)(\omega)\in K_w\}$ is an element of the product $\sigma$-algebra on $K\times \Omega$. Lemma \ref{borelK} completes the proof.
\end{proof}

\begin{remark}[On particle paths]\label{Xmeasurability} 
For each $x\in K$, $t\mapsto X_{0,t}(x)$ is a {\cadlag} function, and $\omega\mapsto X_{0,\cdot}(x)$ is a random variable in the space of {\cadlag} $K$-valued paths (with the usual weak topology). 
Proof of this result is no more than a long exercise in manipulating the definitions and is not included in this article; see \cite{F2012}.
\end{remark}



\section{Proof of the phase transitions}
\label{phaseproofsec}

In this section we prove the results that were stated in Section \ref{ressec}, namely Theorems \ref{tauthm}, \ref{phasethm} and Corollary \ref{taucor}. We begin with some results based on Fubini's theorem and the spatial structure of $K$.

\begin{lemma}\label{f2}If $\sum r_n=\infty$ then $\P\l[\lambda_t(\mc{D}_t)=0\r]=1$.
\end{lemma}
\begin{proof}
Fix $t>0$. By Fubini's theorem, which applies by Lemma \ref{xwmeas},
\begin{align}\label{eq0a1}
\E\l[\lambda(\mc{D}_t)\r]&=\E\l[\int_K\1\{x\in \mc{D}_t\}\lambda(dx)\}\r]=\int_{K} \P\l[x\in\mc{D}_t\r]\lambda(dx).
\end{align}
The rate at which each point $x\in K$ is affected by reproduction events is $\sum_{n\in\N_0} r_n$. Hence, if $\sum r_n=\infty$ then each $x\in K$ has almost surely been involved in some reproduction event during $[0,t]$ and thus $\P\l[x\in\mc{D}_t\r]=0$.  By \eqref{eq0a1}, $\P\l[\lambda(\mc{D}_t)=0\r]=1$.
\end{proof}

\begin{lemma}\label{finatomsdust}
Suppose $A_t$ is finite. Then either $\mc{D}_t=\emptyset$ or, for some $w\in W_*$, $K_w\sw \mc{D}_t$. In the latter case, $\lambda_t(\mc{D}_t)>0$.
\end{lemma}
\begin{proof}
Suppose that $A_t$ is finite and $\mc{D}_t$ is non-empty. Enumerate $A_t$ as $\{[x_i]_t\-i=1,\ldots,N\}$ where $x_i\in K$ are such that $[x_i]_t\neq [x_j]_t$ for $i\neq j$. By definition of $[x]_t$ we have $[x_i]_t\cap [x_j]_t=\emptyset$ for $i\neq j$ and by Lemma \ref{xeqcomp} we have $[x_i]_t=K^{x_i,0,t}_{w_1}$ for all $i$. 

Let $\mc{N}=\max\{|w_1^{x_i,0,t}|\-i=1,\ldots, N\}$ and note that by {\Ka} we have 
$K=\uplus_{w\in W_\mc{N}}K_w.$ 
Since $\mc{D}_t\neq\emptyset$, by \eqref{AtDteq} there is some $x\in K\sc\l(\cup A_t\r)=K\sc(\cup_i K^{x_i,0,t}_{w_1})$. Combined with $K=\uplus_{w\in W_\mc{N}}K_w.$ 
, this means that for some $w'\in W_\mc{N}$ we must have 
$K_{w'}\sw  K\sc(\cup_i K^{x_i,0,t}_{w_1})$, and by \eqref{AtDteq} we must have $K_{w'}\sw \mc{D}_t$. It follows immediately from Lemma \ref{Kdlemma} that $\lambda(\mc{D}_t)>0$.
\end{proof}

\begin{lemma}\label{finglobalratedust}
If $\sum \mc{S}^n r_n=\infty$ then ${\P\l[\exists t>0, \exists w\in W_*, K_w\sw\mc{D}_t\r]=0}.$
\end{lemma}
\begin{proof}
Since $W_*$ is countable, the lemma follows if we can show that $\P\l[\exists t>0, K_w\sw\mc{D}_t\r]=0$ for an arbitrary $w\in W_*$. So fix $w\in W_*$, and set $n=|w|$. The rate at which $K_w$ is affected by reproduction events is
\begin{align*}
\int_{W^*\times K}\1\{K_w\cap K_{w'}\neq\emptyset\}\mc{P}(dw',dp)&=\sum\limits_{w'\in W_*} r_{|w|}\1\{K_w\cap K_{w'}\neq\emptyset\}\\
&\geq \sum\limits_{w'\in W_*} r_{|w|}\1\{K_{w'}\sw K_{w}\}
\end{align*}
Now, by {\Ka}, $K_{w'}\sw K_{w}$ if and only if $w'=wv$ for some $v\in W_*$. Hence,
{\allowdisplaybreaks 
\begin{align*}
\sum\limits_{w'\in W_*} r_{|w|}\1\{K_{w'}\sw K_{w}\}&=\sum\limits_{v\in W_*} r_{|wv|}=\frac{1}{\mc{S}^n}\sum\limits_{m=n}^\infty\mc{S}^m r_m=\infty
\end{align*}
}
It follows immediately that (with probability one) $K_w$ is affected by a reproduction event during $(0,t]$ for any $t>0$. Hence $\P\l[\exists t, K_w\sw\mc{D}_t\r]=0$.
\end{proof}

The following pair of lemmas, which play a crucial part in our arguments, begin the connection between $\mc{D}_t$, $A_t$ and $B^t$.

\begin{lemma}\label{Btnlemma} For each fixed $t>0$, $(B_n^t)_{n\geq 0}$ is a GWVE. The initial state is $B^t_0=\1\{\mc{E}_\emptyset>t\}$ and the $n^{th}$ stage offspring distribution is Binomial with $\mc{S}$ trials and success probability $e^{-tr_{n+1}}$. 
\end{lemma}
\begin{proof} Note first that $e^{-r_nt}$ is the probability that $K_w$, where $|w|=n$, does not see its clock ring during $(0,t]$. If $w\in\mathscr{B}^t_n$ and $|w|=n$, then the (conditional) probability that $wi\in\mathscr{B}^t_{n+1}$ is just $e^{-r_{n+1}t}$. The clocks corresponding to $K_{wi}$ and $K_{wj}$ are independent if $i\neq j$, thus the offspring distribution of $w\in\mathscr{B}^t_n$ is Binomial with $\mc{S}$ trials and success probability $e^{-r_{n+1}t}$. 
\end{proof}

\begin{lemma}\label{dustsizebranching}
For each $t>0$, $\mc{D}_t=\cap_{n\in\N_0}\cup_{w\in\mathscr{B}^t_n} K_w.$
\end{lemma}
\begin{proof}
Suppose first that $x\in \mc{D}_t$, so as $[x]_t$ is a singleton. If $\mc{E}_w\leq t$ for some $K_w\ni x$ then by definition of $X$ the set $X_{0,t}(K_w)$ would be a single point. By Lemma \ref{Kdlemma} the set $K_w$ is infinite and by Lemma \ref{xeqcomp} $K_w\ni x$ implies $K_w\sw[x]_t$, so in fact we must have $\mc{E}_w>$ for all $K_w\ni x$. Thus $x\in \cap_n\cap_{w\in \mathscr{B}^t_n} K_w$.

Similarly, if $x\in \cap_n\cap_{w\in \mathscr{B}^t_n} K_w$ then $\mc{E}_w>t$ for all $K_w\ni x$. If $x\notin \mc{D}_t$ then by Lemma \ref{xeqcomp} we would have $X_{0,t}^{-1}(x)=K^{x,0,t}_{w_1}$, which implies that $\mc{E}_{w_1}\leq t$, in contradiction to the above. Hence we must have $x\in\mc{D}_t$.
\end{proof}


\subsection{Degeneracy of GWVEs}
\label{realanalysissec}

We make use of Lemma \ref{dustsizebranching} through a result of \cite{J1976} which characterises the extinction criteria of particular GWVEs. Define 
\begin{align}
m^t_n&=\E\l[B^t_n\r]=\mc{S}^ne^{-t\sum_0^nr_j},\label{mtndef}\\
g^t&=\sum\limits_{n=1}^\infty \frac{(1-e^{-r_{n+1}t})^{\mc{S}}+\mc{S}e^{-r_{n+1}t}-1}{\mc{S}e^{-r_{n+1}t}\,m^t_n}.\label{gtdef}
\end{align}
Note that $m^t_n\in(0,\infty)$. For all $x> 0$ and $n\in\N$, $n\geq 2$, it holds that $(1-x)^n+nx-1> 0$, and hence $g^t\in(0,\infty]$. 

\begin{lemma}[\citealt{J1976}]\label{l3}
$\P\l[\exists n\in\N, B^t_n=0\r]<1$ if and only if both $\inf m^t_n>0$ and $g^t<\infty$.
\end{lemma}
\begin{proof}
Lemma 1.1 of \cite{J1976} tells us that if $\inf_n m^t_n=0$, then $\P\l[\exists n\in\N, B^t_n=0\r]=1$. Our offspring distributions are binomial with $\mc{S}$ trials, which is sufficient for Theorem 2.3 of \cite{J1976} to apply. From this we obtain that, if $\inf_n m^t_n>0$, then $\P\l[\exists n\in\N, B^t_n=0\r]=1$ if and only if $g^t=\infty$.
\end{proof}


\begin{lemma}\label{a1}
Suppose that $v\in(0,\infty)$ is such that $\inf _n m_n^{v}>0$. Then for all $u\in(0,v)$, $\inf_n m^u_n>0$ and $g^u<\infty$.

\end{lemma}
\begin{proof}
Let $\inf _n m_n^{v}>0$. Suppose that $\epsilon>0$ and that for infinitely many $n$ we have $\frac{1}{n}\sum_0^nr_j>\frac{\epsilon+\log\mc{S}}{v}$. For such $n$, 
\begin{align*}
m^v_n&=\l(\exp\l(\log\mc{S}-v\frac{1}{n}\sum_0^nr_j\r)\r)^n\leq \l(\exp\l(\log\mc{S}-v\frac{\epsilon+\log\mc{S}}{v}\r)\r)^n=e^{-\epsilon n}.
\end{align*}
This may not occur for infinitely many $n$ since $\inf m^v_n>0$. So, we may assume that both $\inf_n m^v_n>0$, and
\begin{equation}\label{pppp101}
\limsup_n \frac{1}{n}\sum_0^nr_j\leq\frac{\log\mc{S}}{v}.
\end{equation}

Let $u\in(0,v)$. By the above, for some $\epsilon>0$ we have $0<u\limsup_n \frac{1}{n}\sum_0^nr_j<\log \mc{S}-\epsilon.$ 
Hence, there exists $N\in\N$ (dependent on $\epsilon$) such that for all $n>N$, $0<u\frac{1}{n}\sum_0^nr_j<\log \mc{S}-\epsilon.$ 
Thus,
\begin{equation*}
m^u_n=\l(\exp\l(\log\mc{S}-u\frac{1}{n}\sum_0^nr_j\r)\r)^n\geq\l(\exp\l(\epsilon\r)\r)^n,
\end{equation*}
so clearly $\inf m^u_n>0$. Also
\begin{align*}
g^u\leq \sum_1^\infty \frac{\mc{S}}{\mc{S}e^{-r_{n+1}u}m^u_n}=\mc{S}\sum_1^\infty\frac{1}{m^u_{n+1}}\leq \mc{S}\sum_1^\infty\frac{1}{(e^{\epsilon})^{n+1}}<\infty
\end{align*}
as required. 
\end{proof}

\begin{lemma}\label{an3}
There exist $C_1,C_2\in(0,\infty)$ (dependent only upon $\mc{S}$) such that for all $t>0$,
$$C_1\sum_n\frac{e^{-r_{n+1}t}}{m^t_n}\leq g^t\leq C_2\sum_n\frac{e^{-r_{n+1}t}}{m^t_n}.$$
\end{lemma}
\begin{proof}
Let $f_n:(0,\infty)\to(0,\infty)$ be the function $f_n(x)=(1-x)^n+nx-1.$ It is elementary to show that there exists $C_1,C_2\in(0,\infty)$ (dependent only upon $\mc{S}$) such that for all $x\in (0,1]$, $C_1x^2\leq f_{\mc{S}}(x)\leq C_2x^2.$ Since $g^t=\sum_n\frac{f_{\mc{S}}(e^{-r_{n+1}t})}{\mc{S}e^{-r_{n+1}t}\,m^t_n},$
the stated result follows.
\end{proof}


\begin{lemma}\label{a3} For each $(r_n)$, precisely one of the following three cases occurs.
\begin{enumerate}
\item $\limsup\limits_{n\to\infty}\frac{1}{n}\sum_0^nr_j=\infty$ if and only if $\inf m^t_n=0$ for all $t>0$.

\item $\limsup\limits_{n\to\infty}\frac{1}{n}\sum_0^nr_j=0$ if and only if $\inf m^t_n>0$ for all $t>0$.

\item $\limsup\limits_{n\to\infty}\frac{1}{n}\sum_0^nr_j\in(0,\infty)$ if and only if both (1)
if $t<t_0$ then $\inf_n m^t_n>0$ and (2)
if $t\geq t_0$ and $\inf m^{t_0}>0$ then $g^{t_0}=\infty$.

\end{enumerate}
\end{lemma}
\begin{proof} Since $\mc{L}=\limsup_{n}\frac{1}{n}\sum_0^nr_j$ exists in $[0,\infty]$ precisely one of $\mc{L}=0$, $\mc{L}=\infty$ and $\mc{L}\in(0,\infty)$ occurs. We give each case in turn.

\textsc{Case 1:} Suppose that $\limsup_{n}\frac{1}{n}\sum_0^nr_j=\infty$. For any $t>0$, we can pick a subsequence $(r_{i_n})$ of $(r_n)$ such that for all $n$, $\frac{1}{i_n}\sum_0^{i_n}r_j>\frac{\log{\mc{S}}+1}{t}$. Hence $m^t_{i_n}\leq (\exp(-1))^{i_n}$ for all $n$, and since $i_n\to\infty$ it follows that $\inf_n m^t_n=0$. 

Conversely, if $\limsup_{n}\frac{1}{n}\sum_0^nr_j=C<\infty$ then for $t=\frac{\log\mc{S}}{2C}>0$ we note that 
$$m^t_n=\l(\exp\l(\log{\mc{S}}-\frac{\log\mc{S}}{2}\frac{1}{C}\frac{1}{n}\sum_0^nr_j\r)\r)^n.$$
For sufficiently large $n$, $\frac{1}{n}\sum_0^nr_j\leq \frac{3}{2}C$, and hence for sufficiently large $n$, $m^t_n\geq \l(\exp\l(\frac{1}{3}\log \mc{S}\r)\r)^n$. Hence $\inf_n m^t_n>0$.

\textsc{Case 2:} Suppose that $\limsup_n\frac{1}{n}\sum_0^n r_j=0$ and let $t>0$. Then, for all sufficiently large $n$ we have $\frac{1}{n}\sum_0^n r_j\leq\frac{1}{2t}$. Hence, for all sufficiently large $n$, $m^t_n\geq \l(\exp\l(\frac{1}{2}\log \mc{S}\r)\r)^n$. Thus $\inf_m m^t_n>0$.

Conversely, suppose that $\inf _n m_n^{t}>0$ for all $t$. Fixing $t$, and using the first step of the proof of Lemma \ref{a1}, we obtain from \eqref{pppp101} that $\limsup_n \frac{1}{n}\sum_0^nr_j\leq\frac{\log\mc{S}}{t}$. However, we have $\inf _n m_n^{t}>0$ for all $t>0$, so $\limsup_n\frac{1}{n}\sum_0^n r_j=0$. 

\textsc{Case 3:} Suppose that $\limsup_n\frac{1}{n}\sum_0^n r_j\in(0,\infty)$. Recall from \eqref{t0def} that $t_0=\frac{\log\mc{S}}{\mc{L}}$ where $\mc{L}=\limsup_{n}\frac{1}{n}\sum_0^nr_j$. Consider first when $s<t_0$. Then there exists $\epsilon\in(0,\log|S|)$ such that $s\leq \frac{\log|S|-\epsilon}{\mathscr{L}}.$ Hence, 
$$m^s_n\geq\l(\exp\l(\log|S|-\frac{\log|S|-\epsilon}{\mathscr{L}}\frac{1}{n}\sum_1^n r_j\r)\r)^n.$$
There exists $N\in\N$ such that for all $n> N$, $\frac{1}{n}\sum_1^n r_j\leq \frac{\log|S|-\epsilon/2}{\log|S|-\epsilon}\,\mathscr{L}.$
Hence, for all $n>N$,
$$m^s_n\geq \l(\exp\l(\log|S|-(\log|S|-\epsilon)\frac{\log|S|-\frac{\epsilon}{2}}{(\log|S|-\epsilon)}\r)\r)^n=\l(\exp(\epsilon/2)\r)^n$$
Thus $m^s_n\to\infty$ and hence $\inf m^s_n>0$.

Now consider $t_0$ itself. Define $a_n\in\R$ by
$r_n=\mathscr{L}+a_n.$
In this notation $m^{t_0}_n=\exp\l(-t_0\sum_1^n a_j\r)$. We now consider two cases.

Firstly, if $\limsup_n \sum_1^n a_j=\infty$ then it is immediate that $\inf_n m^{t_0}_n=0$. Since $m^t_n$ is a decreasing function of $t$ (for each fixed $n\in\N_0$), this implies that $\inf_nm^t_n=0$ for all $t\geq t_0$.

It remains only to consider the case $\limsup \sum_1^n a_j<\infty$, in which case $\inf m^{t_0}_n>0$. By Lemma \ref{an3} there exists $C\in(0,\infty)$ such that $g^{t_0}\geq C(\inf_n m^{t_0}_n)^{-1}\sum_1^n e^{-r_{n+1}t_0}.$
Since $\limsup\frac{1}{n}\sum_1^n r_j\in(0,\infty)$, $(r_n)$ has a subsequence $(r_{i_n})$ such that $\limsup_n r_{i_n}<\infty$. Hence $g^{t_0}=\infty$. By Lemma \ref{a1} this implies that $\inf_n m^t_n=0$ for all $t>t_0$, which completes the proof. 
\end{proof}


\subsection{Continuity} 
\label{tausec}

In this section we establish that various aspects of our model are, in some sense, continuous across time.

For $s<t$, define $\mc{F}_{s,t}=\sigma(M_{(s,t]})$. Recall that in Section \ref{proofintro} we showed that the GWVE $B^t$ is equivalent to an inhomogeneous percolation on the tree $W_*$. For a possibly random $\mc{F}$-measurable time $t\in[0,\infty)$, let $\mc{E}^t=\inf\{s>t\-M_{(t,s]\times\{w\}}\neq \emptyset\}.$ In words, this is the first time after $t$ at which $K_w$ sees a coalescence event. For each $w\in W_*$, $s\in (0,\infty)$ and possibly random $\mc{F}$-measurable time $t$ define 
\begin{align*}
Q_{w,t}&=\{\exists\text{ a sequence }(i_n)_{n\in\N}\sw S\text{ such that }\forall m\in\N, \mc{E}_{wi_1\ldots i_m}>t\}\\
Q^t_{w,s}&=\{\exists\text{ a sequence }(i_n)_{n\in\N}\sw S\text{ such that }\forall m\in\N, \mc{E}^t_{wi_1\ldots i_m}>s\}\\
R^t_w&=\Omega\sc\cup_{s>0} (Q^t_{w,s}).
\end{align*}
When we wish to use a random time $t$ in the above definition we will say so explicitly. For the remainder of this section, the symbol $t$ is used only for deterministic times.

In the language of percolation, $Q_{w,s}$ is the event that $w\in W_*$ is connected to infinity at time $t$. The set $Q^t_{w,s}$ is the event that a connection between $w$ and infinity that exists at time $t$ will continue to exist until (at least) time $t+s$. The set $R^t_{w}$ is the event that any connection between $w$ and infinity which might exist at time $t$ will be instantaneously disconnected immediately after time $t$. 

\begin{lemma}\label{Rt01}
Let $t\in[0,\infty)$ and $w\in W_*$. Then $\P\l[R^t_w\r]$ is either $0$ or $1$.
\end{lemma}
\begin{proof}
If $s_1\leq s_2$ then $Q^t_{w,s_2}\sw  Q_{w,s_1}^t$. Thus, for all $N\in\N$ we have $R^t_w=\Omega\sc\cup_{n\geq N}Q^t_{w,1/n}$. Noting that $Q^t_{w,s}$ is $\mc{F}_{(s,t+s]}$ measurable, we obtain that $R^t_{w}$ is $\mc{F}_{t,t+\frac{1}{N}}$ measurable for all $N\in\N$. The stated result then follows from the Kolmogorov zero-one law.
\end{proof}

\begin{lemma}\label{philemma}
The function $t\mapsto \phi(t)=\P\l[\forall n\in\N_0, B^t_n\neq 0\r]$ is strictly monotone decreasing over $[0,\infty)$. Further, $\phi$ is left continuous over $t\in[0,\infty)$. If $\phi(s)>0$ then $\phi$ is right continuous on $[0,s)$.
\end{lemma}
\begin{proof}
Note that $B^t_n\sw B^s_n$ for all $s\leq t$; it follows immediately that $\phi(t)$ is decreasing. The time at which clock $\mc{E}_w$ ring has a continuous distribution on $[0,\infty$), so for all $[a,b]\sw[0,\infty)$ there is positive probability of having $\mc{E}_w\in[a,b]$. It follows from this that $\phi(t)$ is strictly decreasing.

For continuity, note that $\phi(t)=\lim_{n}\P\l[B^t_n\neq 0\r]$, which is a decreasing limit as $n\to\infty$. Each $\mc{E}_w$ has continuous distribution, so the definition of $B^t$ implies that the function $t\mapsto \P\l[B^t_n\neq 0\r]$ is continuous in $t$. Thus $\phi(t)$ is upper semicontinuous. Since $\phi(t)$ is also decreasing, $\phi(t)$ is left continuous on $[0,\infty)$.

Let $0\leq t<s$ and be such that $\phi(s)>0$. In order to show that $\phi$ is right continuous at $t$, we must prove that the event
\begin{align}\label{cont13}
\{\forall n\in\N_0, B^t_n\neq 0\text{ and }\forall u>t\;\exists n\in\N_0, B^u_n=0\}
\end{align}
has probability zero. Note that 
\begin{align}
\eqref{cont13}&=Q_{\emptyset,t}\cap (\cap_{u>t}(\Omega \sc Q_{\emptyset,u}))=Q_{\emptyset,t}\cap (\cap_{u>0}(\Omega \sc Q_{\emptyset,t+u}))\notag\\
&=Q_{\emptyset,t}\cap (\Omega\sc\cup_{u>0}\,Q_{\emptyset,t+u})\sw Q_{\emptyset,t}\cap (\Omega\sc\cup_{u>0}\,Q^t_{\emptyset,u})\notag\\
&=Q_{\emptyset,t}\cap R^t_\emptyset.\label{cont15}
\end{align}

Suppose \eqref{cont13} has positive probability. Then by \eqref{cont15} we have $\P\l[R^t_\emptyset\r]>0$, which by Lemma \ref{Rt01} implies that $\P\l[R^t_\emptyset\r]=1$. By the time homogeneity of our model this means that also $\P\l[R^0_\emptyset\r]=1$. Hence $\P\big[Q^0_{\emptyset,u}\big]=\P\l[Q_{\emptyset,u}\r]=0$ for all $u>0$, which means that $B^u$ is almost surely degenerate for all $u>0$, in contradiction to our hypothesis that $\phi(s)>0$. So in fact $\P\l[A\r]=0$, which completes the proof.
\end{proof}

\begin{lemma}\label{tault0}
If $\limsup_{n}\frac{1}{n}\sum_0^nr_j\in(0,\infty)$ then $\P\l[\pi<t_0\r]=1$.
\end{lemma}
\begin{proof}
Suppose that $\limsup_{n}\frac{1}{n}\sum_0^nr_j\in(0,\infty)$. By combining Lemmas \ref{a1} and \ref{an3} we have that $\P\l[\forall n\in\N_0, B^{t_0}_n\neq 0\r]=0$. Hence, by the left continuity proved in Lemma \ref{philemma}, for all $\epsilon>0$ there is some $\delta>0$ such that $\P\l[\forall n\in\N_0, B^{t_0-s}_n\neq 0\r]<\epsilon$ for all $s\in[0,\delta)$. Hence $\pi<t_0$ almost surely.
\end{proof}

Let $(\mc{G}_t)$ denote the usual augmented filtration (see section II.67 of \citealt{RW2000}) of $\mc{F}_t=\sigma(M_{[0,s]}\-s\leq t)$. Let
$\pi=\inf\{t>0\-\exists n\in\N, B^t_n=0\}$ and note that, since $(\mc{G}_t)$ is right continuous, $\pi$ is a $(\mc{G}_t)$ stopping time. 

\begin{lemma}\label{piextinct}
If $\limsup_n\frac{1}{n}\sum_0^n r_j<\infty$ then $\P\l[\exists n\in\N_0, B^\pi_n=0\r]=1$.
\end{lemma}
\begin{proof}

Let $\limsup_n\frac{1}{n}\sum_0^n r_j<\infty$ and suppose (for a contradiction) that $\P\l[\forall n\in\N_0, B^\pi_n\neq 0\r]>0$. By definition of $\pi$, almost surely for all $t>\pi$ there is some $n\in\N_0$ such that $B^t_n=0$. Thus
\begin{equation}\label{cont17}
\{\forall n\in\N_0, B^\pi_n\neq 0\text{ and }\forall u>\pi\;\exists n\in\N_0\;B^t_n=0\}
\end{equation}
has positive probability. The same rearrangement as was used in \eqref{cont15}, with $\pi$ in place of $t$, shows that $\eqref{cont17}\sw Q_{\emptyset,\pi}\cap R^\pi_\emptyset.$ Hence $\P\l[R^\pi_\emptyset\r]>0$. 

By the strong Markov property of the time-homogeneous process $M$ at the $(\mc{G}_t)$ stopping time $\pi$, we have that $R^\pi_\emptyset$ and $R^0_\emptyset$ have the same distribution, hence also $\P\l[R^\pi_\emptyset\r]>0$. By Lemma \ref{Rt01} we thus have $\P\l[R^0_{\emptyset}\r]=1$. Thus $\P\big[Q^0_{\emptyset,u}\big]=\P\l[Q_{\emptyset,u}\r]=0$ for all $u>0$, which means that $B^u$ is almost surely degenerate for all $u>0$. 

However, since $\limsup_n\frac{1}{n}\sum_0^n r_j<\infty$, by Lemmas \ref{a1} and \ref{a3} there is almost surely some $\epsilon>0$ such that $\P\l[\forall n\in\N_0, B^\epsilon_n\neq 0\r]>0$, so we reach a contradiction and conclude that in fact, $\P\l[\forall n\in\N_0, B^\pi_n\neq 0\r]=0$.
\end{proof}

\subsection{Dust and GWVEs}
\label{dustandgwves}

In this section we build on Lemma \ref{dustsizebranching} and relate the behaviour of $\mc{D}_t$ to the behaviour of $B^t$. 

\begin{lemma}\label{l2}
Let $t>0$. Then $\mc{D}_t=\emptyset$ if and only if $\exists n\in\N, B^t_n=0.$
\end{lemma}
\begin{proof} Fix $t>0$. Suppose first that for some (random) $n\in\N$, $B^t_n=0$. Then, by Lemma \ref{dustsizebranching}, $\mc{D}_t=0$. For the converse, If $B^t_n\neq 0$ for all $n$ then it is easily seen that there exists a sequence $(w(n))_{n\in\N}$ such that $w(n)\in \mathscr{B}^t_n$ and $K_{w(n)}\supseteq K_{w(n+1)}$. By Lemma \ref{Kwclosed}, each subcomplex $K_w$ is a closed subset of $K$. It follows from this and completeness of $K$ that $\cap_{n\in\N}K_{w(n)}$ is non-empty. By Lemma \ref{dustsizebranching}, $\cap_{n\in\N}K_{w(n)}\sw\mc{D}_t$ so the proof is complete. \end{proof}

\begin{remark}
The argument above has appeared many times in various guises in the random fractals literature, see e.g. Lemma 8 of \cite{D2009}. However, if $K$ is not completely segregated then (Lemma \ref{yinKst2} does not apply so as) $K_w$ might not be closed and $\cap_{n\in\N}K_{w(n)}$ could be empty. We address this issue in Section \ref{Osec}.

As a consequence of Lemma \ref{l2}, $\tau=\pi$. The reason for making the distinction between $\tau$ and $\pi$ will become clear in Section \ref{llproofs}. \end{remark}

\begin{lemma}\label{coveringK2} 
Let $t>0$. If $\mc{D}_t=\emptyset$ then $A_t$ is finite.
\end{lemma}
\begin{proof}If $\mc{D}_t$ is empty then, by Lemma \ref{l2} there is some $n\in\N_0$ such that $B^t_n=0$. Let $N=\min\{n\in\N\-B^t_n=0\}$. Since $\mc{D}_t=\emptyset$, by Lemma \ref{xeqcomp}, for all $x\in K$ we have $[x]_t=K^{x,s,t}_{w_1}$. 

Suppose that $|w^{x,s,t}_1|>N+1$. Then there is some $w\in W_{N+1}$ such that $K_w\supseteq K^{x,s,t}_{w_1}$. Since $|w|>N$ we have $w\notin \mathscr{B}^t$ so there must be some $w'\in W_*$ such that $|w'|\leq|w|<|w^{x,s,t}_1|$ and $K_{w'}\supseteq K_w\supseteq K^{x,s,t}_{w_1}$ with $\mc{E}_{w'}\leq t$. Since then $x\in K_{w'}$ and $|w'|\leq |w^{x,s,t}_1|$, this contradicts the definition of $w^{x,s,t}_1$. Hence in fact $|w^{x,s,t}_1|\leq N+1$.

Therefore, $A_t$ is a subset of $\{K_w\-w\in W_{N+1}\}$, which implies that $A_t$ is finite.
\end{proof}


\begin{lemma}\label{l1}
Let $t>0$. Then $\P\big[\lim_{n}B^t_n=\infty\text{ or }\exists n, B^t_n=0\big]=1.$
Further, almost surely, $\mc{D}_t\neq\emptyset$ if and only if $\lim\limits_{n\to\infty}B_n^t=\infty$.
\end{lemma}
\begin{proof}
To prove the first statement, we use the result of Theorem 1 in \cite{J1974}, which is a restatement (with minor correction) of a result in \cite{C1967}. 

The probability of a individual at stage $n$ in the process $B^t$ having exactly one offspring is given by
$p_{n1}^t=\mc{S}e^{-r_nt}(1-e^{-r_nt})^{\mc{S}-1}$. Note that for $a\in[0,1]$ and $n\geq 1$, $a(1-a)^n\leq \frac{1}{n+1}(1-\frac{1}{n+1})^n$. Since $\mc{S}\geq 2$ we have $p_{n1}^t\leq \l(1-1/\mc{S}\r)^{\mc{S}-1}<1.$ Hence $\sum_n(1-p_{n1}^t)=\infty$, and from \cite{J1974} we have 
$\P\l[\lim_{n}B^t_n=\infty\text{ or }\exists n\geq\mc{N}, B^t_n=0\r]=1.$ It follows immediately from this and Lemma \ref{l2} that $\lim_{n}B^t_n=\infty$ is almost surely equivalent to $\mc{D}_t\neq\emptyset$.
\end{proof}

\begin{lemma}\label{dustsumrnfinite}
Let $t>0$. If $\sum r_n<\infty$ then $\P\l[\mc{D}_t=\emptyset\text{ or }\lambda(\mc{D}_t)>0\r]=1$.
\end{lemma}
\begin{proof}
The process $n\mapsto B^t_n/ \E\l[B^t_n\r]$ is a discreet parameter, non-negative martingale. By the martingale convergence theorem there is some random variable $L^t$ such that $\frac{B^t_n}{\E\l[B^t_n\r]}\to L^t$ almost surely. 

Recall that in \eqref{mtndef} we gave a formula for $\E\l[B^t_n\r]$. Since $\sum r_n<\infty$,
\begin{align}
\E\l[B^t_n\r]&\geq \mc{S}^n\l(\E\l[B^t_0\r]\exp\l(-t\sum_0^\infty r_j\r)\r)=C\mc{S}^n .\label{ppp64}
\end{align}
where $C=C(t)>0$. In the language of \cite{BD1992}, \eqref{ppp64} means that $B^t$ is uniformly supercritical. Since the offspring distribution of $B^t$ is uniformly bounded (by $\mc{S}$), Theorem 2 of \cite{BD1992} applies. In our notation this means that
\begin{equation}
\l\{B^t_n\to\infty\r\}=\l\{L^t >0\r\}.\label{ppp65}
\end{equation}

Now, suppose $\omega\in\mc{A}$ and that $\mc{D}_t\neq\emptyset$. By Lemma \ref{dustsizebranching}, for all $n\in\N$ we have $B^t_n=|\mathscr{B}^t_n|\geq 1$. By the first part of Lemma \ref{l1} it follows that (almost surely) $B^t_n\to\infty$ as $n\to\infty$. By \eqref{ppp65}, $\lim_{n}\frac{B^t_n}{\E\l[B^t_n\r]}>0.$
From this and \eqref{ppp64},
$\liminf_{n}\frac{B^t_n}{C\mc{S}^n}>0$,
where $\liminf_n \frac{B^t_n}{C\mc{S}^n}$ could potentially be infinite. In fact, though, $B^t_n\leq B^t_0\mc{S}^n\leq \mc{S}^n$ so $\liminf_n \frac{B^t_n}{C\mc{S}^n}$ is finite. We write $l=\liminf\limits_{n\to\infty}\frac{B^t_n}{C\mc{S}^n}$ where $l\in (0,\infty)$ (note $l$ is random). Then there exists $N\in\N$ such that for all $n>N$, $\frac{B^t_n}{C\mc{S}^n}\geq l/2.$ So for all $n>N$ we have $B^t_n\geq \frac{Cl}{2}\mc{S}^n$.

Note that the sets $\bigcup_{w\in \mathscr{B}^t_n}K_w$ are decreasing as $n$ increases. By Lemma \ref{dustsizebranching},
$\lambda(\mc{D}_t)=\lim_{n}\lambda\big(\bigcup_{w\in \mathscr{B}^t_n}K_w\big).$
Recall that $\lambda(K_w)>0$ and $\lambda(K_w)$ depends only on $|w|$ (by Lemma \ref{Kdlemma} and {\Kg}  respectively). Hence $\lambda(K_w)=\frac{\lambda(K)}{\mc{S}^n}$. By {\Ka}, 
$\lambda\l(\cup_{w\in \mathscr{B}^t_n}K_w\r)=|\mathscr{B}^t_n|\lambda (K_w)=\frac{B^t_n}{\mc{S}^n}\lambda(K).$
Thus, for $n>N$, $\lambda\l(\cup_{w\in \mathscr{B}^t_n}K_w\r)\geq C l \lambda(K)>0.$
The result follows by Lemma \ref{dustsizebranching}.
\end{proof}



\begin{lemma}\label{tau2}
If $\limsup\frac{1}{n}\sum_0^nr_j=0$ then for all $t\in\R$, $\P\l[\tau>t\r]>0$.
\end{lemma}
\begin{proof}
If $\limsup\frac{1}{n}\sum_0^nr_j=0$ then by Lemma \ref{a3} we have $\inf_n m^t_n>0$ for all $t>0$. By Lemma \ref{a1} we thus have $g^t<\infty$ for all $t>0$ and by Lemma \ref{a1} we thus have $\P\l[\forall n\in\N, B^t_n\neq 0\r]>0$. By Lemma \ref{l2} we thus have $\P\l[\mc{D}_t>0\r]>0$ for all $t\in\R$. Since $\P\l[\mc{D}_t>0\r]>0$ for all $t\in \R$ we also have $\P\l[\tau>t\r]>0$ for all $t\in\R$. 
\end{proof}

\subsection{Proof of Theorems \ref{tauthm},\ref{phasethm} and Corollary \ref{taucor}} 

Let us begin by proving Theorem \ref{phasethm}. Note that the criteria given for our five phases in terms of $\mc{S}$ and $(r_n)$ assign possible each choice of $\mc{S}$ and $(r_n)$ to precisely one phase. Therefore, it suffices to show that the criteria for each phase are sufficient. Let us begin by covering the supercritical phase.

\textbf{Supercritical.} Suppose that $\limsup\frac{1}{n}\sum_0^nr_j=\infty$. We need to show that $\tau=0$ almost surely. By Lemma \ref{a3}, $\inf_n m^s_n=0$ for all $s>0$, and thus from Lemma \ref{l3} we have that $\P[\exists n\in\N, B^s_n=0]=1$. By Lemma \ref{l2} $\P\l[\mc{D}_s=\emptyset\r]=1$ for all $s>0$ and by \eqref{Dmon} we have $\P\l[\forall s>0, \mc{D}_s=\emptyset\r]=1$. Hence $\P\l[\tau=0\r]=1$.

We will now give the arguments for the four remaining phases. Note that, in the lower/upper subcritical and semicritical phases, Lemma \ref{tau2} tells us that $\P\l[\tau>0\r]=1$ and $\P\l[\tau>t\r]>0$ for all $t\in\R$. The behaviour of $\tau$ in the critical phase is discussed below.

\textbf{Lower Subcritical.} Suppose that $\sum\mc{S}^nr_n<\infty$ and consider $t\in(0,\tau)$. Hence, by definition of $M$, for each (deterministic) $s<\infty$ only finitely many coalescence events occur during $[0,s]$. Since $t<\tau<\infty$, almost surely only finitely many coalescence events have occurred during $[0,t]$. By Lemma \ref{xeqcomp}, $A_t$ is a finite set, which combined with Lemma \ref{finatomsdust} implies that $\lambda(\mc{D}_t)>0$.

\textbf{Upper Subcritical.} Suppose that $\sum\mc{S}^nr_n=\infty$ and $\sum r_n<\infty$. Consider $t\in(0,\tau)$. If $A_t$ was finite then, since $\mc{D}_t\neq\emptyset$, by Lemma \ref{finatomsdust} the set $\mc{D}_t$ must contain a subcomplex of $K$; but Lemma \ref{finglobalratedust} implies that this is not the case. Hence in fact $A_t$ must be infinite. By Lemma \ref{dustsumrnfinite}, $\lambda(\mc{D}_t)>0$ almost surely.

\textbf{Semicritical.} Suppose that $\sum r_n=\infty$ and $\limsup\frac{1}{n}\sum_0^nr_j=0$. Consider $t\in(0,\tau)$. The same argument applies here as given above in the upper subcritical case to show that $A_t$ is infinite. However, in this case Lemma \ref{f2} tells us that $\P\l[\lambda(\mc{D}_s)=0\r]=1$ for fixed $s\in(0,\infty)$. By \eqref{Dmon}, in fact $\P\l[\forall s>0, \lambda(\mc{D}_s)=0\r]$.

\textbf{Critical.} Suppose that $\limsup\sum_0^nr_j\in(0,\infty)$ and write $\mc{L}=\limsup\sum_0^nr_j$. Consider first if $s<t_0$. Then by combining Lemmas \ref{l3} and \ref{an3} we have that $\P\l[\exists n\in\N B^s_n=0\r]<1$, so as by Lemma \ref{l1}, 
$\P\l[\mc{D}_s\neq\emptyset\r]>0\text{ for }s<t_0.$ Hence $\P\l[\tau>s\r]\in(0,1)$ for all $s<t_0$.

Similarly, by Lemma \ref{l3}, $\P\l[\exists n\in\N, B^{t_0}_n=0\r]=1$ and hence by Lemma \ref{l1}, $\P\l[\mc{D}_{t_0}=\emptyset\r]=1$. By \eqref{Dmon} we then have $\P\l[\mc{D}_{s}=\emptyset\r]=1$ for all $s\geq t_0$.

By Lemma \ref{l2} we have $\P\l[\tau=\pi\r]=1$, so by Lemma \ref{tault0} we have $\P\l[\tau<t_0\r]=1$. 
Now consider $t\in(0,\tau)$. The same argument as in the semicritical case tells us that $A_t$ must be infinite and that $\P\l[\forall s>0, \lambda({D}_s)=0\r]=1$. This completes the proof of Theorem \ref{phasethm}. \hfill\qed\vspace{1ex}

We now prove Theorem \ref{tauthm}. The first part of the statement of Theorem \ref{tauthm} is a trivial consequence of Lemma \ref{coveringK2}; if $t>\tau$ then $\mc{D}_t=\emptyset$ almost surely and, by Lemma \ref{coveringK2}, $A_t$ is almost surely finite. 

For the second statement, by Theorem \ref{phasethm} we have that $\P\l[\tau=0\r]=1$ if and only if our model is supercritical. So, assume our model is not supercritical. By Theorem \ref{phasethm} there is some (deterministic) $\delta>0$ such that $\P\l[\mc{D}_\delta\neq\emptyset\r]>0$, so as by Lemmas \ref{philemma} and \ref{l2} we have that $t\mapsto \P\l[\mc{D}_t\neq\emptyset\r]$ is continuous on $[0,\delta)$. Note that for all $\epsilon>0$, $$\P\l[\tau=0\r]\leq \P\l[\tau\leq \epsilon\r]=\P\l[\mc{D}_\epsilon=\emptyset\r]=1-\P\l[\mc{D}_\epsilon\neq\emptyset\r],$$ 
which by continuity tends to $1-\P\l[\mc{D}_0\neq\emptyset\r]=0$ as $\epsilon\downarrow 0$. Hence $\tau>0$ almost surely. \hfill\qed\vspace{1ex}

We now prove Corollary \ref{taucor}. By Lemma \ref{l2} we have $\tau=\pi$ almost surely and combining this with Lemma \ref{piextinct} we obtain that $\mc{D}_\tau=0$ almost surely, provided $\limsup_n\frac{1}{n}\sum_0^nr_j<\infty$. By Theorem \ref{phasethm}, if our model is not supercritical then $\limsup_n\frac{1}{n}\sum_0^n r_j<\infty$, which completes the proof. \hfill\qed\vspace{1ex}


\section{The case $\mc{O}\neq\emptyset$}
\label{Osec}

We now describe the modifications required to prove our results in the case where $K$ is a segregated space but potentially not a completely segregated space (i.e. we allow $\mc{O}\neq \emptyset$). Essentially, the difference in this case is that Lemma \ref{Kwclosed} breaks down. We used Lemma \ref{Kwclosed} in precisely two places, namely the proofs of Lemma \ref{yinKst2} and Lemma \ref{l2}. It is these two lemmas which we will `repair' to deal with the case $\mc{O}\neq \emptyset$. We will replace them, respectively, with the following.

\begin{lemmaa}\label{yinKst3}
Almost surely, for all $0<s<t$, all $x\in K$ and all $n\in\N$, if $(E_m^{x,s,t})$ is infinite then $\lim\limits_{m\to\infty}p_m^{x,s,t}\in K_{w_n}^{x,s,t}$.
\end{lemmaa}

\begin{lemmab}\label{l22}
Almost surely, for all $t\in [0,\infty)$, $\mc{D}_t=\emptyset$ if and only if $B^t_n=0$ for some $n\in\N_0$.
\end{lemmab}

Lemmas {\refyinKst3} and {\refl22} are the same (respectively) as Lemmas \ref{yinKst2} and \ref{l2}, except for the presence in both cases of the `almost surely'. Their proofs are given in Section \ref{llproofs}.

Let $\mc{N}$ be the null set of $\Omega$ on which the `almost surely' in Lemma {\refyinKst3} does not hold. This is the null set (that we mentioned in Section \ref{modelformaldef}) on which we wish to define $X$ differently in the case $\mc{O}\neq\emptyset$. So, fix some point $x^*\in K$ and from now, for all $s<t$ define
\begin{equation}\label{pp4}
X_{s,t}=\Big\{
\begin{array}{ll}
\text{as in }\eqref{pp3} & \text{ if }\omega\in\Omega\sc\mc{N}\\
x^* & \text{ if }\omega\in\mc{N}.
\end{array}
\end{equation}
In words, when $\omega\in\mc{N}$ the flow instantaneously (and at all times) moves all the particles to the point $x^*$. Thus $\tau=0$ on $\mc{N}$.

The arguments in the proof of Theorem \ref{existthm}, using Lemma {\refyinKst3} in place of Lemma \ref{yinKst2}, work as before so long as $\omega\in\Omega\sc\mc{N}$. On the other hand, for $\omega\in\mc{N}$ it is readily seen that \eqref{pp4} trivally implies the conclusions of Theorem \ref{existthm}. Thus Theorem \ref{existthm} remains true.

Essentially the same principle applies to using Lemma {\refl22} in place of Lemma \ref{l2}; the results in Section \ref{phaseproofsec} that were stated for all $\omega\in\Omega$ and rely upon Lemma \ref{yinKst2} and/or \ref{l2} now hold only almost surely. It is a simple matter to check that this is sufficient to make the proof of Theorems \ref{tauthm},\ref{phasethm} and Corollary \ref{taucor} go through as before.

\subsection{Proof of Lemmas {\refyinKst3} and {\refl22}}
\label{llproofs}

Recall Remark \ref{K4im}, that {\Kc} was immediately if $\mc{O}=\emptyset$. In the case $\mc{O}\neq\emptyset$ it is {\Kc} that fills the gap, as the following arguments show.

We first prove Lemma {\refyinKst3}. Let $\mathscr{E}=\{(x,s,t)\in K\times(0,\infty)^2\-|E^{x,s,t}|=\infty\}$
and define an equivalence relation on $\mathscr{E}$ by 
$(x,s,t)\sim (x',s',t')$ if and only if $E^{x,s,t}=E^{x',s',t'}.$
Let $[x,s,t]$ denote the equivalence class of $(x,s,t)$ in $\mathscr{E}$ under $\sim$. In view of this definition, we write $E^[x,s,t]=E^{x,s,t}$ and similarly for $(u^{x,s,t}_m, w^{x,s,t}_m, p^{x,s,t})$. It follows from {\Ka} and Definition \ref{Eseq} that $E_1^{x,s,t}=E_1^{x',s',t'}$ if and only if $(x,s,t)\sim (x,'s',t')$. 

Let $(\hat{s}_k,\hat{t}_k)_{k\in \N}$ be a countable dense subset of $\{(s,t)\in\R^2\-s\leq t\}$ and let $(\hat{x}_j)_{j\in\N}$ be such that for all $w\in W_*$ there is some $x_j\in K_w$. Note that since $E_1^{x,s,t}<E_2^{x,s,t}$ for all $x,s,t$, for each $(x,s,t)\in\mathscr{E}$ there is some $k$ such that $E^{x,s,t}=E^{x,\hat{s}_k,\hat{t}_k}$. For fixed $k$, from Definition \ref{Eseq} we have $E_1^{x,\hat{s}_k,\hat{t}_k}=E_1^{x',\hat{s}_k,\hat{t}_k}$ if and only if $w_1^{x,\hat{s}_k,\hat{t}_k}$, hence for all $x\in K$ and all $k\in\N$ there is some $j\in\N$ such that $E^{x,s,t}=E^{\hat{x}_j,\hat{s}_k,\hat{t}_k}$. Thus,
\begin{align}
&\l\{\exists (x,s,t)\in\mathscr{E}, \exists n\in\N, \lim_{m\to\infty} p^{x,s,t}_m\notin K^{x,s,t}_{w_n}\r\}\label{cont23}\\
&\hspace{7pc}=\bigcup\limits_{j,k,n\in\N}\l\{p^{[\hat{x}_j,\hat{s}_k,\hat{t}_k]}_m\notin K^{[\hat{x}_j,\hat{s}_k,\hat{t}_k]}_{w_n}\r\}.\label{cont22}
\end{align}

For fixed $j,k$ and $n$, each $p^{[\hat{x}_j,\hat{s}_k,\hat{t}_k]}_m$ is sampled (independently) from within $K^{[\hat{x}_j,\hat{s}_k,\hat{t}_k]}_{w_m}$ according to $\lambda$. So, by {\Kc}, with probability $1/\mc{S}$ we have 
\begin{equation}\label{cont21}
p_m^{[\hat{x}_j,\hat{s}_k,\hat{t}_k]}\in K_{w_{m+1}}^{[\hat{x}_j,\hat{s}_k,\hat{t}_k]}\sw\overline{K_{w_{m+1}}^{[\hat{x}_j,\hat{s}_k,\hat{t}_k]}}\sw K^{[\hat{x}_j,\hat{s}_k,\hat{t}_k]}_{w_m},
\end{equation}
By the Borel-Cantelli lemma, \eqref{cont21} occurs for infinitely many $m\in\N$, with probability one. So, for each $n\in\N$ almost surely we can find some $m\geq n$ for which \eqref{cont21}
 holds, implying that 
$\lim_{r\to\infty}p_r^{[\hat{x}_j,\hat{s}_k,\hat{t}_k]}\in\overline{K_{w_{m+1}}^{[\hat{x}_j,\hat{s}_k,\hat{t}_k]}}\sw K^{[\hat{x}_j,\hat{s}_k,\hat{t}_k]}_{w_n}.$
Thus $\P\big[p^{[\hat{x}_j,\hat{s}_k,\hat{t}_k]}_m\notin K^{[\hat{x}_j,\hat{s}_k,\hat{t}_k]}_{w_n}\big]=0$. Combining this with \eqref{cont22} we have
that \eqref{cont23} is a $\P$-null, which completes the proof. \hfill\qed\vspace{1ex}

We give the proof of Lemma {\refl22} in two parts, the first of which is the following lemma. Thanks to Lemma {\refyinKst3}, all the results stated in Sections \ref{phaseproofsec}-\ref{tausec} are available to us, with the caveat that results which previously held for all $\omega\in\Omega$ now hold only almost surely.

\begin{lemma}\label{l222}
$\P\l[\forall t\in \Q,\text{ if }\forall n\in\N\;B^t_n\neq 0\text{ then }\mc{D}_t\neq 0\r]=1$.
\end{lemma}
\begin{proof}
Since $\Q$ is countable it suffices to prove the result for a single fixed $t\in \Q$. In fact, we can prove for any fixed $t\in\R$, as follows. Suppose that $Q=\{\forall n\in\N, B^t_n \neq 0\}$ has positive probability (and note that if this has probability zero then there is nothing to prove). Let $\P_Q$ denote the conditional measure of $\P$ on the event $Q$.

By Lemma 4.10 of \cite{L1992}, the distribution of $B^t$ under $\P_Q$ is that of a GWVE with $n^{th}$ stage offspring distribution given by
\begin{equation}\label{cont30}
f^*_n(s)=\frac{f_n(q_n+(1-q_n)s)-q_{n-1}}{1-q_{n-1}},
\end{equation}
where $q_n=\P\l[\exists m\geq n B^t_m=0\,|\, B^t_n=1\r]$ and $f_n(s)$ is the generating function of the offspring distribution in stage $n$ of $B^t$. Since $q_n=\lim_{N\to\infty}f_n\circ \ldots \circ f_{n+N}(0)$ we have $f^*_n(0)=0$ which in turn means that, under $\P_Q$, each individual in $B^t$ has at least one child.

This tells us the behaviour of $|\mathscr{B}^t|$ under $\P_Q$, but we need a little extra work to describe $\mathscr{B}^t$ itself. The clocks $\{\mc{E}_w\-w\in W_n\}$ are all independent and identically distributed (under $\P$). Therefore, under $\P_Q$, the number $k$ of elements of $W_n$ which are in $\mathscr{B}^t_n$ is given by \eqref{cont30}, but precisely which such elements of $W_n$ is given by the uniform distribution on the set of subsets of $W_n$ of size $k$. 

In view of this description, define a sequence $(w_n)_{n\in\N_0}$ as follows. Set $w_0=\emptyset$ and note that $w_0\in B^t_0$ $\P_Q$-almost surely. Now, if $w_n$ is defined and $w_n\in B^t_n$ $\P_q$-almost surely, enumerate the set of children of $w_n$ in $\mathscr{B}^t_{n+1}$ as $C_{n+1}\sw W_{n+1}$. Independently of all else, sample $w_{n+1}$ uniformly from $C_{n+1}$.

By our description of $B^t$ under $\P_Q$, using {\Kc} we have that with probability (at least) $1/\mc{S}$,
$\overline{K_{w_{n+1}}}\sw K_{w_n}.$
The offspring distributions of each individual in $\mathscr{B}^t_n$ are independent, hence by the Borel-Cantelli lemma there is almost surely an infinite subsequence $(\wt{w}_n)$ of $(w_n)$ such that $\overline{K_{\wt{w}_{n+1}}}\sw K_{\wt{w}_n}$ for all $n$. Thus, $\cap_n K_w=\cap_n K_{\wt{w}_n}$ is non-empty. It follows from this and the `almost sure' replacement of Lemma \ref{dustsizebranching} (see our comments before Lemma \ref{l222}) that, almost surely, $\mc{D}_t\supseteq \bigcap_n K_{\wt{w}_n}\neq\emptyset$.\end{proof}

Our task now is to upgrade Lemma \ref{l222} into Lemma {\refl22}. By Lemma \ref{l222}, almost surely, for all $q\in Q$, $\tau<q$ if and only if $\pi<q$. Since $\Q$ is dense in $\R$ it follows immediately that $\tau=\Q$ almost surely. Recall that $B^t_n\leq B^t_s$ for $s\leq t$. It follows from this and Lemma \ref{piextinct} that, almost surely, for all $t\in\R$, $\exists n\in\N_0 B^t_n=0$ if and only if $t\geq \pi$. Hence, by (the almost sure version of) Lemma \ref{dustsizebranching} we have $\mc{D}_\pi=\emptyset$. Since $\tau=\pi $ almost surely we have $\mc{D}_{\tau}=\mc{D}_\pi$ almost surely and thus $\P\l[\mc{D}_\tau=\emptyset\r]=1$. Therefore, using \eqref{Dmon} we have that almost surely, $\mc{D}_t=0\;\iff\;t\geq \tau\;\iff\;t\geq\pi\;\iff\;\exists n\in\N_0, B^t_n=0.$ \hfill\qed\vspace{1ex}


\section{The Hausdorff dimension of the dust}
\label{nendsec}

It is natural to ask further questions about the non-empty null dust in the semicritical and critical phases. According to Lemma \ref{dustsizebranching} and our comments following \eqref{iiiii7}, $\mc{D}_t$ is a random fractal. In fact, $\mc{D}_t$ belongs to the large class of random fractals which are, in some sense, stochastic generalizations of iterated function systems (IFSs). See \cite{F2003} for an introduction to fractal geometry and IFSs. 

IFSs have been generalised in many directions, both deterministically and stochastically, and formulas for the Hausdorff dimension of the corresponding attractors have been obtained in increasing generality; see \cite{D2009}, \cite{M2010} and the references therein. Generality sufficient to cope with $\mc{D}_t$, at least in terms of Hausdorff dimension, seems to have been reached only recently and a result corresponding to the Hausdorff measure of $\mc{D}_t$ does not seem to be known. 

Let $||\cdot||$ denote the Euclidean norm on $\R^d$, and let $\mc{L}^d$ denote $d$ dimensional Lebesgue measure. Let $A^\circ$ denote the (topological) interior of the set $A$, and let the diameter of $A$ be given by $\diam(A)=\sup\{||x-y||\-x,y\in E\}$. Recall that a similarity $f$ is a function between subsets of $\R^d$ such that for some $\eta\in (0,\infty)$ and all $x,y$, $||f(x)-f(y)||=\eta||x-y||$. We write $\eta=\lip(f)$. Recall also that $\dim_\mc{H}(A)$ denotes the Hausdorff dimension of $A$ (for $A\sw\R^d$ this is with respect to the metric $(x,y)\mapsto||x-y||$).


In order to link our results to those of \cite{D2009}, we must impose some extra assumptions on $K$. 

\begin{defn}
We say $K$ is D-compatible if $K\sw\R^d$, $D_K=||\cdot||$, and 
\begin{enumerate}
\item For all $w\in W_*$, $K_w$ is compact.
\item For all $w\in W_*$ and $i\in S$ there exists a similarity  $f^{(w,i)}:K_w\to K_{wi}$. 
\item There exists $\epsilon,\epsilon'\in (0,1)$ and a sequence $(l_n)\sw[\epsilon,\epsilon']$ such that for all $w\in W$, $\lip(f^{(w,i)})=l_{|w|}.$
\item There exists $\kappa>0$ such that for all $w\in W_*$, $\mc{L}^d(K_w^\circ)\geq\kappa\diam(K_w)^d$.
\end{enumerate}
\end{defn}



\begin{normaltheorem}\label{dimDt}
Suppose that $K$ is D-compatible and that $\P\l[\mc{D}_t\neq \emptyset\r]>0$. Let $\mathscr{L}=\limsup_{n}\frac{1}{n}\sum_1^nr_j$ and $\mathscr{S}=\limsup_{n}\frac{1}{n}\sum_1^n(-\log l_n)$.
Conditional on $\{\mc{D}_t\neq\emptyset\}$,
$$\dim_\mc{H}(\mc{D}_t)=\l(\frac{\log|S|-t\mathscr{L}}{\mathscr{S}}\r)\,\vee\, 0.$$
By setting $r_n=0$, it follows that $\dim_\mc{H}(K)=\frac{\log|S|}{\mathscr{S}}$.
\end{normaltheorem}
\begin{proof}
For each $s>0$ and $n\in\N$ let $\alpha_{s,n}^t=|S|(l_n)^se^{-r_{n+1}t}$ and $\rho^t(s)=\liminf_{n}\frac{1}{n}\sum_{j=1}^n\log \alpha_{s,j}$. Using {\Ka} and the D-compatibility conditions we apply Theorem 1 of \cite{D2009}, which yields that, if $\mc{D}_t\neq\emptyset$, then $\dim_\mc{H}(\mc{D}_t)=\sup\{s\in [0,\infty)\-\rho^t(s)>0\}$.

By Theorem \ref{phasethm}, if $\P\l[\mc{D}_t\neq\emptyset\r]>0$ then $0\leq\limsup_n\frac{1}{n}\sum_1^nr_j<\infty$. Note also that by \textit{3} of the D-compatability conditions, $0\leq -\log \epsilon'\leq -\log(l_n)\leq -\log\epsilon<\infty$. A short calculation shows that
\begin{align*}
\rho^t(s)
&=\log |S|-t\limsup\limits_{n\to\infty}\l(\frac{1}{n}\sum_1^nr_j\r)-s\limsup\limits_{n\to\infty}\l(\frac{1}{n}\sum_1^n(-\log l_n)\r).
\end{align*}
The result follows.
\end{proof}

\textit{This article is identical in all but typographical detail to the version which is to appear in the Annals of Probability. I would like to thank the referee for several comments that greatly improved the presentation of this article.}

\bibliographystyle{plainnat}
\bibliography{biblio1}
\end{document}